\documentclass[a4paper,12pt]{article}
\usepackage{amssymb,amsmath,amsthm}
\usepackage[left=.75 in, right=.75 in,top=.75 in, bottom=.75 in]{geometry}
\usepackage{graphicx}
\usepackage{tikz}
\usepackage{float}
\usepackage{enumitem}
\usepackage{changepage}
\usepackage{hyperref}
\usepackage{subcaption}
\usepackage{titling}

\def\H1{{_0}H^1(\Omega)}

\def\N{\mathbb{N}}

\def\R{\mathbb{R}}

\def\XXint#1#2#3{{\setbox0=\hbox{$#1{#2#3}{\int}$ }
\vcenter{\hbox{$#2#3$ }}\kern-.6\wd0}}

\newtheorem{lem}{Lemma}[section]
\newtheorem{cor}[lem]{Corollary}
\newtheorem{prop}[lem]{Proposition}
\newtheorem{thm}[lem]{Theorem}
\newtheorem{remark}[lem]{Remark}

\newtheorem*{prop*}{Proposition}
\newtheorem*{thm*}{Theorem}
\newtheorem*{def*}{Definition}
\newtheorem*{lem*}{Lemma}

\numberwithin{equation}{section}

\title{Existence for a Cahn-Hilliard Model for Lithium-Ion Batteries with Exponential-Growth Boundary Conditions}
\author{Kerrek Stinson \\  kstinson@uni-bonn.de \\ Hausdorff Center for Mathematics, Universit\"at Bonn}

\begin{document}
\pagestyle{plain}
\setcounter{page}{1}

 \begin{titlingpage}
    \maketitle
    \begin{abstract}
The Cahn-Hilliard reaction model, a nonlinear, evolutionary PDE, was introduced to model phase separation in lithium-ion batteries. Using Butler-Volmer kinetics for electrochemical consistency, this model allows lithium-ions to enter the domain via a nonlinear Robin-type boundary condition $\partial_\nu \mu = R(c,\mu)$ for the chemical  potential $\mu$, with $c$, the lithium-ion density. Importantly, $R$ depends exponentially on $\mu$. Fixed point methods are applied to obtain existence of regular solutions of the Cahn-Hilliard reaction model in dimension $3,$ allowing for recovery of exponential boundary conditions as in the physical application.
    \end{abstract}
    

\end{titlingpage}

\newpage

\section{Introduction}

Lithium-ion batteries are a critical technological tool, used everyday in cellphones, laptops, and cars, but they are also closely interwined with the push towards renewable energies, which requires energy storage for intermittent power sources such as the sun or wind. In recognition of the current eminence of this technology, the 2019 Nobel Prize in chemistry was awarded to Goodenough, Whittingham, and Yoshino for pioneering development in the theory and scalable industrialization of lithium-ion batteries \cite{nobelPrize}. Despite their ubiquity, batteries are still far from reaching their limits. One prominent challenge associated with lithium-ion batteries is a reduced life cycle resulting from phase transitions occurring during use:  During the process of charging, lithium-ions intercalate into the host material of the cathode heterogeneously, giving rise to regions of either high or low lithium-ion density with very little in between. This process is referred to as phase separation and induces a strain on the material, which eventually degrades the cathode. Damage of the cathode's host material diminishes battery performance and shortens the life-cycle (see \cite{Bazant-Theory2013}, \cite{Dal2015-Comp}, and references therein).

Understanding the formation of phase transitions is essential to improvement in battery design and performance. From the perspective of modeling, much has been done in this direction using diffuse interface phase field models. Phase field models are often governed by variational principles arising from a global energy functional which depends on regular inputs (e.g. Sobolev functions). As noted in \cite{Bazant-Theory2013}, the phase field field model is robust, allowing for electrochemically consistent models for the time evolution of lithium-ion batteries. Competing models include the shrinking core model and the sharp interface model. However, as noted in Burch et al. \cite{Burch-PhaseTransformation2008}, the shrinking core model fails to capture fundamental qualitative behaviors such as anisotropy. Furthermore, in \cite{HanVanDerVen2004} it is proposed that the phase field model may provide a more accurate numerical analysis of the problem than the sharp interface model, which seeks to model the evolution of the phase boundary as a free boundary problem (see \cite{caginalp1991}; see also \cite{Acharya-msThesis}, and references therein, for benefits of the phase field model).

The Cahn-Hilliard Reaction (CHR) model,
\begin{align}
\label{pde:CHRnoElastic}
\text{CHR model} \ &\left\{ \begin{aligned}
& \partial_t c = \Delta \mu  &\text{in } \Omega \times (0,T),\\
&  \mu = -\Delta c +f'(c)  &\text{in } \Omega \times (0,T),\\
& \partial_\nu c = 0 & \text{on }\partial \Omega \times (0,T), \\
& \partial_\nu \mu = R(c,\mu ) &\text{on }\partial \Omega \times (0,T), \\
& c(0) = c_0 & \text{in } \Omega,
\end{aligned}\right.
\end{align} was introduced in a series a papers by Bazant et al.  (see \cite{Bazant-Theory2013}, \cite{burch2009size}, \cite{cogswell2012coherency}, \cite{singh2008intercalation}, \cite{Bazant-PhaseSepDyn2014}) to model the evolution of lithium-ions within a nanoparticle of a lithium-ion battery. Novelly, this model incorporates Butler-Volmer kinetics, which arise in (\ref{pde:CHRnoElastic}) as the nonlinear, high-order boundary condition $\partial_\nu \mu = R(c,\mu)$. In contrast to the Cahn-Hilliard equation, the CHR model is not mass-preserving and allows for lithium-ions to enter the domain. The purpose of this paper is to prove a short-time existence theorem in dimension $N=3$ for the CHR model with $f$ and $R$ as in the physical application \cite{singh2008intercalation}, that is, given by 
\begin{equation} \label{def:chemPotFunc}
f(s)  := \omega s(1-s)+KT_{\rm{abs}}(s\log(s) +(1-s)\log(1-s)),  \quad \quad s\in [0,1],
\end{equation}  
and
\begin{equation}\label{def:Rbazant}
R(s,w) := R_{\rm{ins}}-R_{\rm{ext}}= k_{\rm{ins}}\exp(\beta (\mu_e-w)) - k_{\rm{ext}}s\exp(\beta(w-\mu_e)), \quad \quad s\in (0,1), \ w\in \R,
\end{equation}
where $\omega, K, T_{\rm abs}, k_{\rm ins}, k_{\rm ext}, \mu_e, \beta>0$ are physical parameters. We do not include a parameter representing phase interface width; though as we are not considering the sharp interface limit (see, e.g., \cite{Abels2009ExistenceOW}, \cite{AlikakosBatesChen}, \cite{garcke-CurvDrivenEvo}, \cite{stinsonLiBatteryGamma}), this can be done.

Given the prevalence of phase separation in many applications, there is a wealth of literature studying Cahn-Hilliard type equations. Existence of strong solutions was first proven by Elliott and Songmu for $f(s) = (s^2-1)^2$ \cite{elliottCH}. Subsequently, Elliott and Luckhaus extended existence to the case of a logarithmic potential (\ref{def:chemPotFunc}) \cite{elliottLuckhaus-logPot}, which still frequently presents a challenge for existence of solutions to related models (see, e.g., \cite{cherfilsMiranvilleZelik}, \cite{davoliScarpaTrussardi}, \cite{garcke-CHlogPot}). A topic of recent interest, and closer in spirit to the CHR model, is that of the Cahn-Hilliard equation with dynamic boundary conditions, which for example could be 
\begin{equation}\label{eqn:dynBC}
\partial_\nu c = -\partial_t c + \kappa \Delta_{\Gamma} c - g'(c) \quad \quad \text{on } \partial \Omega \times (0,T),
\end{equation}
where $ \Delta_{\Gamma}$ is the Laplace-Beltrami operator and $g$ is a surface potential (see \cite{garcke-CHdynamicBC}, \cite{LiuWu-DynamicBC}, and references therein). As can be seen from the Allen-Cahn type equation in (\ref{eqn:dynBC}), these models incorporate phase separation in the bulk and on the surface. To prove existence for this type of model, the Cahn-Hilliard free energy, 
\begin{equation}
I[c] : = \int_\Omega f(c) + \frac{1}{2} \|\nabla c\|^2 \, dx,
\end{equation} is modified to include a surface potential, which gives rise to the above boundary condition when taking the first variation of the energy. Interestingly, boundary conditions emulating Cahn-Hilliard phase separation on the surface can be accounted for by including a surface term in the metric of the gradient flow \cite{garcke-CHdynamicBC}. 

Studying a viscous variant of the Cahn-Hilliard reaction model (\ref{pde:CHRnoElastic}), Kraus and Roggensack developed a variational approach to existence of solutions via a doubly-nonlinear differential inclusion (see also \cite{mielke-doublyNonlinearEvo}), which redefines $\mu$ as $\mu:= -\Delta c +f'(c) + \delta \partial_t c$ for $\delta >0$ \cite{kraus2017}. Using convex duality techniques, Kraus and Roggensack introduce a novel framework for encapsulating the high-order boundary condition within a modified metric emulating the $H^{1}$-dual norm. The $H^{1}$-dual norm gives rise to the gradient flow structure for the Cahn-Hilliard equation (see \cite{fife-CHgradFlow}), and is used in minimizing movements schemes to prove existence of solutions (see also \cite{AGS-GradFlows}). For Kraus and Roggensack, it is no different, and they prove existence of solutions to the viscous CHR model via minimizing movements. Critically the compactness argument, to obtain convergence of the approximate solutions, prevents the passage of the viscosity term $\delta\to 0$ in the limit. Recently, the author introduced an alternative compactness argument to extend this variational approach to the CHR model (\ref{pde:CHRnoElastic}) without viscosity \cite{stinsonExistence}. However, given the energetic nature of the aforementioned approaches, it is unsurprising that these existence theorems are restricted to polynomial growth of $R$. 
 
Herein, we develop an existence and regularity theory via a fixed point method, allowing for recovery of the electrochemically consistent exponential boundary conditions with $R$ given by (\ref{def:Rbazant}) and $f$ given by the thermodynamically relevant density (\ref{def:chemPotFunc}). The primary result is as follows (see Section \ref{sec:prelim} for the definition of $H^{r,s}$).

\begin{thm}\label{thm:CHRexist3D}
Suppose $\Omega \subset \R^3$ is an open, bounded set with smooth boundary. Let $f$ and $R$ be as in (\ref{def:chemPotFunc}) and (\ref{def:Rbazant}), and $c_0\in H^4(\Omega)$ satisfy 
\begin{equation}\label{eqn:initRestrict}
\epsilon < c_0(x)  <1-\epsilon \quad  \text{ for a.e. } \quad x\in \Omega,
\end{equation}
 with 
 \begin{equation}\label{eqn:initCompat}
 \partial_\nu c_0 = 0 \quad \text{and} \quad \partial_\nu (\Delta c_0) = -R(c_0,-\Delta c_0 +f'(c_0)) \quad \text{on} \quad \partial \Omega.
 \end{equation} Then there is a solution $c \in H^{6,1+1/2}(\Omega\times (0,T))$ of the Cahn-Hilliard reaction model (\ref{pde:CHRnoElastic}) for some $T=T(c_0)>0$.
\end{thm}

This is the first result, of which the author is aware, proving existence of a solution for the CHR model with the exponential nonlinearity in the boundary condition, as given by (\ref{def:Rbazant}), and in the physcially relevant dimension $N=3$. The approach adopted herein has the advantage of obtaining existence for high-order nonlinear boundary conditions, but to account for the singularity of $f'$ at $0$ and $1$, we must limit our consideration to initial conditions $c_0$ satisfying (\ref{eqn:initRestrict}).
This is the primary `unphysical' hypothesis required for our result, as the conditions $\partial_\nu c_0 = 0$ and $\partial_\nu (\Delta c_0) = -R(c_0,-\Delta c_0 +f'(c_0))$ describe a necessary compatibility between the boundary conditions and initial condition (see Theorem \ref{thm:SPEwithBC} and Remark \ref{rmk:BCcompat}). It is natural to speculate that the approach to prove Theorem \ref{thm:CHRexist3D} combined with truncation arguments as in \cite{elliottLuckhaus-logPot}, \cite{garcke-CHlogPot} relying on the logarithmic potential could provide the necessary maximum principle to lift assumption (\ref{eqn:initRestrict}); however, in practice, this is complicated by the dependence of $R$ on $\mu$, and thereby $f'$. 
Alternatively, assuming $f\in C^5(\R)$, assumption (\ref{eqn:initRestrict}) may be dropped in the statement of Theorem \ref{thm:CHRexist3D}. This type of assumption is commonly made for Cahn-Hilliard equations, with a classical choice being $f(s):= s^2(1-s)^2$. Finally, as our focus is on recovering solutions to the diffuse interface model with exponential boundary conditions, we only consider constant mobility but imagine that a similar program, with the right parabolic estimates, could account for non-constant mobility (see, e.g., \cite{dai-CHdegenMobility}, \cite{elliott-CHdegenMobility}).

\begin{remark}
Though variational structure can provide a powerful tool, part of the strength of our result is that it does not depend on any underlying variational structure for the PDE. Consequently, Theorem \ref{thm:CHRexist3D} holds for any regular $f \in C^5$ and $R\in C^3$ (domains unspecified) such that the compatibility conditions (\ref{eqn:initCompat}) hold along with the range of $c_0$ and $\mu_0 := -\Delta c_0 + f'(c_0)$ being bounded away from singularities of $f$ and $R$, as in (\ref{eqn:initRestrict}). Thus this also provides existence of regular solutions for short time for Marcus Kinetics (and its generalization, see \cite{Bazant-Theory2013}) with $R$, up to constants, given by $R(\mu) =\exp(-\mu^2)(\exp(-\mu) -\exp(\mu)) $.
\end{remark}

\begin{remark}
We note that Theorem \ref{thm:CHRexist3D} also holds for dimension $N=2$, with modifications to the proof only occuring when using embeddings with dimension dependent parameters. 
\end{remark}

Though we will introduce this notation precisely in the preliminaries, so as to provide a brief discussion of the proof of Theorem \ref{thm:CHRexist3D}, let $H^{r,s}(\Omega \times (0,T)) = H^{r,s}(\Omega_T)$ denote the anisotropic Sobolev space of functions with $r$ derivatives in the spatial domain $\Omega\subset \R^3$ and $s$ derivatives in the time interval $(0,T)$. We will show that there is a solution, $c$, of a truncated version of the CHR model (see (\ref{pde:truncCHR})) belonging to the space $H^{6,1+1/2}(\Omega_T)$, which by trace and embedding theorems will be sufficient to obtain continuity of $\mu$ in space and time and recover the exponential boundary condition encapsulated by $\partial_\nu \mu = R(c,\mu)$. To obtain a solution in this higher regularity Sobolev space, we could hope to prove existence of a strong solution in a lower regularity space, such as $H^{4,1}(\Omega_T)$, and then use the structure of the CHR model to directly bootstrap. However, in Remark \ref{rmk:bootstrapNecessity}, we will see that due to the nonlinearity of the boundary condition such an approach is not feasible without some further insight. To circumnavigate this challenge, we will prove existence by showing that a PDE-defined operator satisfies the hypotheses of Schaefer's fixed point theorem \cite{EvansPDE}. The crux of this argument is showing that any solution of the truncated variant of the CHR model (see (\ref{pde:truncCHR})) satisfies a uniform bound in $H^{5,1+1/4}(\Omega_T)$. Obtaining this bound requires the bulk of our attention.

We briefly outline the contents of this paper. In Section \ref{sec:prelim}, we introduce notation and recall a variety of estimates and embeddings for fractional Sobolev spaces and anistropic Sobolev spaces. In Section \ref{sec:regular}, we jump straight to the heart of our proof and develop the fixed point argument to prove existence of a solution to the CHR model in $H^{5,1+1/4}(\Omega_T)$. Subsequently, in Section \ref{sec:regular2}, a bootstrapping argument recovers the remaining regularity, finding a solution in $H^{6,1+1/2}(\Omega_T)$, and thereby completing the proof of Theorem \ref{thm:CHRexist3D}. In the course of the fixed point argument, we use an a priori estimate regarding the continuity of solutions to a truncated CHR model (see Theorem) whose proof is delayed till Section \ref{sec:aprioriEst}.

\section{Preliminaries}\label{sec:prelim}

We introduce function spaces in Subsections \ref{subsec:fracSobolev} and \ref{subsec:anisoSobolev}, which will be essential in our analysis. Here, we remind the reader of the classical fractional Sobolev spaces in $1$-dimension, but pay special attention to how estimates scale with respect to $T>0$. We then recall a class of anisotropic Sobolev spaces used by Lions and Magenes \cite{LM-v2}. Integrating our knowledge of the two spaces, we use slicing in the temporal variable to introduce an equivalent semi-norm for the anisotropic Sobolev spaces. In Subsection \ref{subsec:parabolic}, we highlight the Gagliardo-Nirenberg inequality and recall an existence and regularity theorem for $4$th-order in-homogeneous parabolic equations. The excited reader may wish to skip the proofs contained in this section and return as necessary.

\subsection{Notation}
We enumerate notation used throughout the paper. We use $C$ to denote a generic constant, which can change from line to line. If dependence of $C$ on parameter $a$ is emphasized, we will denote this by either $C_a$ or $C(a)$. Given a domain $\Omega\subset \R^N$ specified by context, we define $\Gamma : =\partial \Omega.$ We let $\nu$ denote the outward normal of $\Gamma$. We say that a set $\Omega\subset \R^N$ has smooth boundary if, for every $x\in \Gamma$ there is some $r>0$, such that up to rotation, $B(x,r)\cap \Omega$ coincides with the epigraph of a $C^{\infty}(\R^{N-1})$ function. As is standard, we let $p^* := \frac{Np}{N-p}$ be the critical exponent for the Sobolev embedding in dimension $N>p.$ Further for $T>0,$ we define $\Omega_T : = \Omega \times (0,T)$ and set $\Sigma_T : =\Gamma \times (0,T).$ Given Banach spaces $(\mathcal{B}_0,\|\cdot \|_0)$ and $(\mathcal{B}_1,\|\cdot \|_1)$, we denote the continuous embedding of $\mathcal{B}_0$ into $\mathcal{B}_1$ by $\mathcal{B}_0\hookrightarrow \mathcal{B}_1$, and the compact embedding of $\mathcal{B}_0$ into $\mathcal{B}_1$ by $\mathcal{B}_0\hookrightarrow\hookrightarrow \mathcal{B}_1$. Given also $a<b \in \R$, with $J := (a,b)$, let $L^p(a,b; \mathcal{B}_0) = L^p(J; \mathcal{B}_0)$ denote the space of Bocher $p-$integrable functions on $J$ with values in $\mathcal{B}_0.$ For a good resource on such spaces, we refer the reader to \cite{LeoniBook}. Likewise for $T>0$, let $C([0,T];\mathcal{B}_0)$ denote the space of bounded uniformly continuous functions with values in $\mathcal{B}_0$ on the closed interval $[0,T].$

\subsection{Fractional Sobolev Spaces}\label{subsec:fracSobolev}

In the interval $(a,b) \subset \R$, for $u\in H^1(a,b)$ and $s\in (0,1)$, we define a semi-norm via the difference quotient introduced by Gagliardo (see, e.g., \cite{diNezza-hitchhiker}, \cite{gagliardo-fracSobolevSpace}, \cite{LeoniBook}),
\begin{equation}\label{def:fracSeminorm}
|u|_{H^{s}(a,b)} : = \left( \int_a^b \int_a^b \frac{|u(x) - u(y)|^2}{|x-y|^{1+2s}} \ dy \ dx \right)^{1/2}.
\end{equation}
We then define the norm 
\begin{equation}\label{def:fracNorm}
\|u\|_{H^s(a,b)}:=\|u\|_{L^2(a,b)}+|u|_{H^{s}(a,b)}.
\end{equation}
The completion of $H^1(a,b)$ in $L^{2}(a,b)$ with respect to the norm (\ref{def:fracNorm}) is the fractional Sobolev space of order $s$, denoted by $H^s(a,b).$
We prove a result for one-dimensional fractional Sobolev spaces.
\begin{prop}
Suppose $u\in H^{s}(0,T)$ and $\psi \in C^\infty[0,T].$ Then $\psi u \in H^s(0,T)$, and for any $\epsilon \in (0,T)$, it satisfies the bound
\begin{equation}\nonumber
\|\psi u \|_{H^s(0,T)} \leq C((1+\epsilon^{-s})\|\psi\|_\infty  + \epsilon^{1-s}\|\nabla \psi\|_\infty )\|u\|_{L^2(0,T)} + \|\psi\|_\infty|u|_{H^s(0,T)}.
\end{equation}
\end{prop}
\begin{proof}
As the control of the $L^2$ norm of $\psi u$ is straightforward, we estimate the seminorm $|\psi u |_{H^s(0,T)}$ as follows.
\begin{equation}\nonumber
\begin{aligned} 
\int_0^T\int_0^T \frac{|\psi (x)u(x)-\psi(t)u(t)|^2}{|x-t|^{1+2s}} \ dx \ dt \leq & 2\int_0^T \int_0^T |\psi(x)|^2\frac{|u(x)-u(t)|^2}{|x-t|^{1+2s}} \ dx \ dt \\
& +2\iint_{\{|x-t|\leq \epsilon\}} |u(t)|^2\frac{|\psi(x)-\psi(t)|^2}{|x-t|^{1+2s}} \ dx \ dt \\
& +2\iint_{\{|x-t|> \epsilon\}} |u(t)|^2\frac{|\psi(x)-\psi(t)|^2}{|x-t|^{1+2s}} \ dx \ dt.
\end{aligned}
\end{equation}
To bound the terms of the right-hand side, we immediately have
\begin{equation}\nonumber
\int_0^T \int_0^T |\psi(x)|^2\frac{|u(x)-u(t)|^2}{|x-t|^{1+2s}} \ dx \ dt  \leq \|\psi\|_\infty^2 |u|_{H^s(0,T)}^2.
\end{equation}
For the second term, we use the mean value theorem and Fubini's theorem to find
\begin{equation}\nonumber
\begin{aligned}
\iint_{\{|x-t|\leq \epsilon\}} |u(t)|^2\frac{|\psi(x)-\psi(t)|^2}{|x-t|^{1+2s}} \ dx \ dt  \leq & \|\nabla \psi\|_\infty^2 \int_{0}^T|u(t)|^2 \left( \int_{\{x\in (0,T):|x-t|\leq \epsilon\}} |x-t|^{1-2s} \ dx\right) dt  \\
\leq &\|\nabla \psi\|_\infty^2 \int_{0}^T|u(t)|^2\ dt  \int_{-\epsilon}^\epsilon |\zeta|^{1-2s} \ d\zeta \\
= & C \|\nabla \psi\|_\infty^2 \|u\|_{L^2(0,T)}^2 \epsilon^{2-2s}.
\end{aligned}
\end{equation}
The third term is directly estimated as 
\begin{equation}\nonumber
\begin{aligned}
\iint_{\{|x-t|> \epsilon\}} |u(t)|^2\frac{|\psi(x)-\psi(t)|^2}{|x-t|^{1+2s}} \ dx \ dt \leq & C\|\psi\|^2_\infty \epsilon^{-2s} \|u\|_{L^2(0,T)}^2.
\end{aligned}
\end{equation}
Combining the above inequalities, we conclude the lemma.
\end{proof}
As a result of the above lemma and a reflection argument, we obtain the following extension result.
\begin{cor}\label{cor:fracExtension}
Let $0<T\leq T_0$. Suppose that $u\in H^{s}(0,T).$ There exists an extension of $u$, $\tilde{u}\in H^{s}(0,T_0)$, such that
\begin{equation}\label{bdd:fracExtension}
\|\tilde{u} \|_{H^s(0,T_0)} \leq C_{T_0}\Big((1+T^{-s} )\|u\|_{L^2(0,T)} +|u|_{H^s(0,T)}\Big),
\end{equation}
with a constant $C_{T_0}>0$ independent of $T$.
\end{cor}

We prove an estimate that is helpful for comparing the semi-norm of a fractional Sobolev space to that of the standard Sobolev space.
\begin{prop}\label{prop:BesovEst}
Given $u\in H^{1}(0,T)$ and $s\in (0,1),$ we have the following estimate:
\begin{equation}\nonumber
|u|_{H^{s}(0,T)}\leq \frac{1}{s\sqrt{2(1-s)}} T^{1-s}\|\partial_t u\|_{L^2(0,T)}.
\end{equation}
\end{prop}
\begin{proof}
By definition of the semi-norm and the fundamental theorem of calculus,
\begin{equation} \label{def:BesovSemi}
\begin{aligned}
|u|_{H^{s}(0,T)}^2  =& \int_0^T \int_0^T \frac{|u(x) - u(y)|^2}{|x-y|^{1+2s}} \ dy \ dx \\
= &  \int_0^T \left( \left(\int_0^x  + \int_{x}^T\right)\frac{|\int_x^y \partial_t u \ d\sigma |^2}{|x-y|^{1+2s}} \ dy \right) dx.
\end{aligned}
\end{equation}
Fixing $x\in (0,T),$ we bound the $x$ variable's integrand by the change of variables $\bar y = y-x$ and $\bar \sigma = \sigma - x$, and applying Hardy's inequality (see \cite{LeoniBook}). To be precise,
\begin{equation}\nonumber
\begin{aligned}
\int^T_x  \frac{|\int_x^y \partial_t u \ d\sigma |^2}{|x-y|^{1+2s}} \ dy &= \int_0^{T-x} \frac{|\int_x^{x+\bar y} \partial_t u \ d\sigma |^2}{|\bar y|^{1+2s}} \ d\bar y \\
& =\int_0^{T-x} \frac{|\int_0^{\bar y} \partial_t u(x+\bar \sigma) \ d\bar \sigma |^2}{|\bar y|^{1+2s}} \ d\bar y \\
& \leq (1/s)^2\int_0^{T-x}\bar y^{1-2s} |\partial_t u(x + \bar y)|^2 \ d \bar y \\
& = (1/s)^2\int_x^{T} |x-y|^{1-2s} |\partial_t u(y)|^2 \ d y.
\end{aligned}
\end{equation}
By the same argument for the other integral, we obtain
\begin{equation}\nonumber
\begin{aligned}
|u|_{H^{s}(0,T)}^2 \leq & (1/s)^2\int_0^T\int_0^T|x-y|^{1-2s} |\partial_t u(y)|^2 \ dy \ dx \\
= & (1/s)^2\int_0^T\left(\int_0^T|x-y|^{1-2s} \ dx \right) |\partial_t u(y)|^2 \ dy \\
\leq & \left(\frac{1}{s\sqrt{2(1-s)}}\right)^2 T^{2-2s} \|\partial_t u\|_{L^2(0,T)}^2,
\end{aligned}
\end{equation}
concluding the result.
\end{proof}

We lastly make note of a simple lemma, which shows how the semi-norm changes for a rescaled domain.

\begin{lem}\label{lem:fracNormStretch} 
Let $0<T$, $s\in (0,1)$, and $u\in H^s(0,T)$. Define $u_T(x):=u(Tx)$ for $x\in (0,1)$. Then
$$ |u|_{H^s(0,T)} = T^{\frac{1-2s}{2}} |u_T|_{H^s(0,1)}.$$
\end{lem}
\begin{proof}
We obtain
\begin{equation}\nonumber
\begin{aligned}
|u|_{H^s(0,T)}^2 = & \int_0^T \int_0^T \frac{|u(x) - u(y)|^2}{|x-y|^{1+2s}} \ dx \ dy \\
= & \int_0^1 \int_0^1 \frac{|u(Tx) - u(Ty)|^2}{|Tx-Ty|^{1+2s}} \ T dx \ T dy =  T^{1-2s} |u_T|_{H^s(0,1)}^2.
\end{aligned}
\end{equation}
\end{proof}

\subsection{Anisotropic Sobolev spaces}\label{subsec:anisoSobolev}
Let $\Omega\subset \R^N$ be a smooth domain. 
For $(a,b) \subset \R$, $r,s\geq 0$, we define 
\begin{equation}\label{def:anisoSobolevSpace}
H^{r,s}(\Omega \times (a,b)):=L^2(a,b; H^{r}(\Omega))\cap H^s(a,b;L^2(\Omega)),
\end{equation}
where $H^{r}(\Omega)$ and $H^s(a,b;L^2(\Omega))$ are defined via interpolation of Sobolev spaces of integer order (see \cite{LM-v2}).
Note that $H^0(\Omega) = L^2(\Omega)$. Recalling the notation $\Omega_T : = \Omega \times (0,T),$ $\Sigma_T := \Gamma \times (0,T),$ we have
\begin{equation}\nonumber
H^{r,s}(\Omega_T):=L^2(0,T; H^{r}(\Omega))\cap H^s(0,T;L^2(\Omega)).
\end{equation}
Likewise, we may define the anisotropic Sobolev space with domain $\Sigma_T.$ It is standard to endow $H^{r,s}(\Omega \times (a,b))$ with the norm arising from interpolation (denoted by ``$,I$") given by
\begin{equation}\label{def:InterpNorm}
\|u\|_{H^{r,s}(\Omega_T),I} := \left(\int_a^b \|u(\cdot,t)\|_{H^{r}(\Omega),I}^2 \ dt  + \|u\|_{H^s(a,b;L^2(\Omega)),I}^2 \right)^{1/2}.
\end{equation}
As an aside, we recall how interpolation and the interpolation norm are defined. 
\begin{remark}\label{rmk:interp}
Generically, suppose that $X$ and $Y$ are separable Hilbert spaces, with $X$ densely embedded into $Y.$ There is a positive, self adjoint, and unbounded operator $\Lambda$ on $Y$ such that $X = \rm{dom}(\Lambda)$, and the norm on $X$ is equivalent to the graph norm of $\Lambda,$ i.e.,
\begin{equation}\nonumber
\frac{1}{C} \|x\|_X\leq (\|x\|_Y^{2}+\|\Lambda x\|_Y^{2})^{1/2} \leq C \|x\|_X.
\end{equation}
For the construction of such an operator, we refer the reader to \cite{dautrayLions-v3}, \cite{LM-v1}, \cite{rieszNagy-FA}. Using spectral theory for unbounded operators (see \cite{dautrayLions-v3} and references therein), we may consider fractional powers of the operator $\Lambda.$ Then the interpolation space $[X,Y]_\theta$ for $\theta\in [0,1]$ is defined by 
\begin{equation}\label{def:interpSpace}
[X,Y]_\theta := \rm{dom}(\Lambda^{1-\theta}),
\end{equation}
with norm
\begin{equation}\label{def:XYinterpNorm}
\|\cdot\|_{[X,Y]_\theta} := (\|\cdot\|_Y^2 + \|\Lambda^{1-\theta}\cdot\|_Y^2)^{1/2}.
\end{equation}
In the context of Sobolev spaces (see Proposition \ref{prop:anisoInterp}), for example, we have 
\begin{equation}\nonumber
\|\cdot\|_{H^{1/2}(\Omega),I} := \|\cdot\|_{[H^1(\Omega),H^{0}(\Omega)]_{1/2}}. 
\end{equation}
We note that the interpolation space defined by (\ref{def:interpSpace}) is norm equivalent to that defined by the K-method (see, e.g., \cite{LeoniBook}, \cite{LM-v1}). However, the norm (\ref{def:XYinterpNorm}) is more directly related to norms arising from the Fourier transform.
\end{remark}

For $u\in H^{s}(\R;L^2(\Omega))$, we may consider the Fourier transform of $u$ in the variable $t$ given by the Bochner integral
\begin{equation}\nonumber
\hat{u}(\xi) : =\frac{1}{(2\pi)^{1/2}}\int_\R e^{-i\xi t}u(t) \, dt,
\end{equation}
and define the Fourier norm (denoted by ``$,{\rm F}$") on the space $H^{s}(\R; L^2(\Omega))$ (see also \cite{LM-v1}) by
\begin{equation}\nonumber
\|u\|_{H^{s}(\R;L^2(\Omega)),\rm{F}} : =\left(\int_\R (1+|\xi|^{2})^{s}\|\hat{u}(\xi)\|_{L^2(\Omega)}^2 \, d \xi \right)^{1/2}.
\end{equation}

We will need precise results about the extension properties of the function spaces in (\ref{def:anisoSobolevSpace}). Consequently, norms akin to (\ref{def:BesovSemi}) will prove more useful than (\ref{def:InterpNorm}). Let $\|\cdot \|_{H^{r}(\Omega)}$ denote any standard norm choice for $H^{r}(\Omega).$
\begin{lem}
Let $\Omega\subset \R^N$ be an open, bounded set with smooth boundary, and $r\geq 0$. For $s\in (0,1),$ we recall (\ref{def:fracNorm})  to define the norm
 \begin{equation} \label{def:besovAnisoNorm}
 \|u\|_{H^{r,s}(\Omega_T)} := \left(\int_0^T \|u(\cdot,t)\|_{H^{r}(\Omega)}^2 \ dt  + \int_\Omega \|u(x,\cdot)\|_{H^{s}(0,T)}^2 \ dx \right)^{1/2}.
 \end{equation}
 The norm (\ref{def:besovAnisoNorm}) is equivalent to the norm defined by (\ref{def:InterpNorm}). The same is true on domains $\Sigma_T.$
\end{lem}
\begin{remark}
Note the norm (\ref{def:besovAnisoNorm}) is naturally extended to $H^{r,s}(\Omega_T)$ with $s\in (k,k+1)$ for $k\in \N$ by considering the norm $$\left(\|u\|_{H^{r,k}(\Omega_T)}^2 + \int_\Omega \|\partial_t^k u(x,\cdot)\|_{H^{s-k}(0,T)}^2\right)^{1/2}.$$
\end{remark}
\begin{proof}
By classical results, we have that $\|\cdot \|_{H^r(\Omega),I}$ is equivalent to $\|\cdot\|_{H^{r}(\Omega)}$ (see \cite{LeoniBook} for a proof in the case $\Omega = \R^N$; the following argument proves the result for extension domains $\Omega$). Thus, it suffices to take $r=0$ and prove that the norm $\|\cdot\|_{H^{0,s}(\Omega_T)}$ is equivalent to $\|\cdot\|_{H^{0,s}(\Omega_T),I}.$

By Corollary \ref{cor:fracExtension}, there is an extension operator $\mathcal{T}$, defined via reflection and truncation in the variable $t$ (for each fixed point $x\in \Omega$), such that $\mathcal{T}:H^{0,0}(\Omega_{T})\to H^{0,0}(\Omega\times\R)$ and $\mathcal{T}:H^{0,1}(\Omega_T)\to H^{0,1}(\Omega\times\R)$ are linear and bounded. By interpolation (see Theorem 5.1 of Chapter 1 in \cite{LM-v1}), it follows that $\mathcal{T}:H^{0,s}(\Omega_T)\to H^{0,s}(\Omega\times\R)$ is linear and bounded in the topology of the interpolation norm (\ref{def:InterpNorm}). By a direct computation in the spirit of Corollary \ref{cor:fracExtension}, we have that $\mathcal{T}$ is also continuous in the topology defined by the norm (\ref{def:besovAnisoNorm}). Consequently, using the equivalence of the Gagliardo norm and Fourier norm on $\R$ (see \cite{LeoniBook}) and Fubini's theorem, we have 
\begin{equation}\label{bdd:normEquiv}
\begin{aligned}
\|u\|_{H^{0,s}(\Omega_T)}\leq&  C \|\mathcal{T}u\|_{H^{0,s}(\Omega \times \R)}  \\
\leq & C \|\mathcal{T}u\|_{H^{s}(\R;L^2(\Omega)),\rm{F}}  \leq   C \|\mathcal{T}u\|_{H^{s}(\R;L^2(\Omega)),I} \leq C \|u\|_{H^{s}(0,T;L^2(\Omega)),I}
\end{aligned}
\end{equation}
(see subsection 7.1 of Chapter 1 in \cite{LM-v1} for equivalence of Fourier and interpolation norm).

To obtain the reverse inequality to prove equivalence of the norms, we may essentially reverse the sequence of inequalities in (\ref{bdd:normEquiv}). In the same way that $\mathcal{T}$ was shown to be continuous, we have that the restriction operator $\pi:u \mapsto u|_{\Omega_T}$ mapping from $H^{0,s}(\Omega\times\R)$ to $H^{0,s}(\Omega_T)$ is continuous in the interpolation norm. Consequently,
\begin{equation}\nonumber
\begin{aligned}
\|u\|_{H^{s}(0,T;L^2(\Omega)),I} = &  \|\pi ( \mathcal{T}u)\|_{H^{s}(0,T;L^2(\Omega)),I} \\
\leq & C \| \mathcal{T}u\|_{H^{s}(\R;L^2(\Omega)),I} \leq  C \|\mathcal{T}u\|_{H^{s}(\R;L^2(\Omega)),\rm{F}} \leq C \|\mathcal{T}u\|_{H^{0,s}(\Omega \times \R)} \leq C  \|u\|_{H^{0,s}(\Omega_T)},
\end{aligned}
\end{equation}
where we have, once again, used the equivalence of the Gagliardo and Fourier norms.
\end{proof}
We will make use of some interpolation theorems, which provide regularity of certain quantities. The following result is Proposition 2.1 of Chapter 4 in \cite{LM-v2}.
\begin{prop} \label{prop:anisoInterp}
Let $\Omega$ be an open, bounded set with smooth boundary. For $r,s \geq 0$ and $\theta \in (0,1),$ we have 
\begin{equation}\nonumber
[H^{r,s}(\Omega_T),H^{0,0}(\Omega_T)]_{\theta} = H^{(1-\theta)r,(1-\theta)s}(\Omega_T).
\end{equation}
The same is true on domains $\Sigma_T.$
\end{prop}

In Theorem \ref{thm:traceSobo}, we quote a special case of Theorem 4.2 in Chapter 4 of \cite{LM-v2}, which shows that Hilbert space valued Sobolev functions can be embedded into a space of continuous functions. Although we refer to this theorem as the ``trace Theorem" because it provides the machinery to define traces for many Sobolev spaces, including anisotropic Sobolev spaces in Theorem \ref{thm:LMnormalReg}, we will often use this theorem to obtain $L^\infty(0,T)$ bounds.

\begin{thm}[Trace Theorem]\label{thm:traceSobo} Let $X$ and $Y$ be separable Hilbert spaces, with $X$ densely embedded into $Y$, and $s>1/2$. Then the embedding $ L^2(0,T;X) \cap H^{s}(0,T; Y) \hookrightarrow C([0,T];[X,Y]_{1/(2s)})$ holds. 
\end{thm}

\begin{thm}\label{thm:LMnormalReg}
Let $\Omega\subset \R^N$ be an open, bounded set with smooth boundary. Let $u\in H^{r,s}(\Omega_T)$ with $r>1/2,$ $s\geq 0.$ If $j$ is an integer such that $0\leq j<r-1/2$, we may define the $j$th normal derivative $\partial_\nu^j u \in H^{\mu_j,\lambda_j}(\Sigma_T)$, where
\begin{equation}\label{def:LMnormalRegParams}
\frac{\mu_j}{r} = \frac{\lambda_j}{s} = \frac{r-j-1/2}{r}.
\end{equation}
Furthermore, the map $u \mapsto (\partial_\nu^j u)_{\{0\leq j<r-1/2\}} $ is surjective, continuous, with continuous right inverse.
\end{thm}
\begin{proof}
This is primarily a restatement of Theorem 2.1 in Chapter 4 of \cite{LM-v2}. Existence of the continuous right inverse follows from Theorem 3.2 of Chapter 1 in \cite{LM-v1}.
\end{proof}

\begin{prop}\label{prop:intermediateReg}
Let $\Omega\subset \R^N$ be an open, bounded set with smooth boundary. Suppose that $u\in H^{k,k/4}(\Omega_T)$, $k\in \N$.  Then, $\nabla^2 u \in H^{k-2,(k-2)/4}(\Omega_T)$, with the map $u\mapsto \nabla^2 u $ continuous in the respective topologies.
\end{prop}
\begin{proof}
Let $\mathcal{T}:H^k(\Omega)\to H^k(\R^N)$ be a linear extension operator (defined via reflection and a partition of unity as in Theorem 13.4 and Remark 13.5 of \cite{LeoniBook}), such that there is $r>0$ with ${\rm{supp}}\,\mathcal{T}(\xi) \subset B(0,r)$ for all $\xi\in H^k(\Omega).$  We extend $u$ as $\tilde{u}(x,t):=\mathcal{T}(u(\cdot,t))(x).$ Since $\mathcal{T}(u(\cdot,t))$ is defined by reflecting $u(\cdot,t)$ near $\Gamma$ and using a partition of unity (see, e.g., Theorem 13.17 in \cite{LeoniBook} or Theorem 5.4.1 \cite{EvansPDE}), the regularity of $\tilde{u}$ in time is preserved, and so $\tilde{u}\in H^{k,k/4}(\R^N \times (0,T))$ with 
\begin{equation}\label{bdd:spatExt}
\|\tilde{u}\|_{H^{k,k/4}(\R^N \times (0,T))}\leq  C\|u\|_{H^{k,k/4}(\Omega_T)}.
\end{equation} Write $x = (x',x_N)\in \R^{N-1}\times \R$. By identifying $\xi \in H^{k,k/4}(\R^N\times (0,T))$ with the function $x_N \mapsto \xi((\cdot,x_N),\cdot),$ as noted in \cite{LM-v2}, we may decompose the anisotropic Sobolev space as 
\begin{equation}\nonumber
\begin{aligned}
 H^{k,k/4} & (\R^N \times (0,T)) = \\
& H^k\left(\R;L^2(\R^{N-1}\times (0,T))\right)\cap L^2\left(\R;H^{k,k/4}(\R^{N-1}\times (0,T))\right).
\end{aligned}
\end{equation}
Consequently, we apply an intermediate derivative theorem (Theorem 2.3 of Chapter 1 in \cite{LM-v1}) and Proposition \ref{prop:anisoInterp} to conclude that $\tilde{u}$ maps continuously to 
\begin{equation}\nonumber
\nabla^2 \tilde{u} \in H^{k-2}(\R;L^2(\R^{N-1}\times (0,T)))\cap L^2(\R;H^{k-2,(k-2)/4}(\R^{N-1}\times (0,T))).
\end{equation}
By continuity of the restriction operator and (\ref{bdd:spatExt}), the proposition is proven.
\end{proof}

\subsection{Estimates and Parabolic Regularity}\label{subsec:parabolic}

We will use the Gagliardo-Nirenberg inequality \cite{nirenbergOnPDE}, which improves upon the more standard Sobolev-Gagliardo-Nirenberg embedding theorem (see, e.g., \cite{LeoniBook}).
\begin{thm}[Gagliardo-Nirenberg inequality]\label{thm:nirenbergIneq}
Suppose that $\Omega\subset \R^N$ is an open, bounded set with Lipschitz boundary. Then the following inequality is satisfied for measurable functions $v:$
\begin{equation}\nonumber
\begin{aligned}
\|\nabla^j v\|_{L^p(\Omega)} \leq & C_1 \|\nabla^m v\|_{L^r(\Omega)}^a\|v\|_{L^q(\Omega)}^{1-a} + C_2\|v\|_{L^q(\Omega)},
\end{aligned}
\end{equation}
with $a \geq 0$ satisfying
\begin{equation}\nonumber
\begin{aligned}
\frac{j}{m} \leq a \leq 1,  \quad &\  \frac{1}{p} = \frac{j}{n} +a\left(\frac{1}{r} - \frac{m}{n}\right) + (1-a)\frac{1}{q}.
\end{aligned}
\end{equation}
\end{thm}

We consider existence of solutions of the linear fourth-order parabolic PDE
\begin{equation} \label{pde:hoSPE} 
\left\{\begin{aligned}
& \partial_t c + \Delta^2 c = g  & \text{in } \Omega_T,\\
& \partial_\nu c  = \alpha & \text{on } \Sigma_T,\\
& \partial_\nu (\Delta c)  = \beta  & \text{on } \Sigma_T,\\
& c(0)  =c_0 & \text{in } \Omega.
\end{aligned}\right.
\end{equation}
For regularity and existence, we have the subsequent theorem, which is a case of Theorem 5.3 of Chapter 4 in \cite{LM-v2}.  
\begin{thm} \label{thm:SPEwithBC}
Let $\Omega\subset \R^N$ be an open, bounded set with smooth boundary and $k\in [1,3/2)\cup \{2\}$. Suppose $g\in H^{k,k/4}(\Omega_T)$, $c_0 \in H^{k+2}(\Omega)$, $\alpha \in H^{\mu_1,\lambda_1}(\Sigma_T)$, and $\beta \in H^{\mu_3,\lambda_3}(\Sigma_T)$, where $\mu_j$ and $\lambda_j$ are defined in (\ref{def:LMnormalRegParams}) with $r = 4+k$ and $s = 1+k/4$. We further assume the compatibility condition $\partial_\nu c_0 = \alpha(\cdot, 0)$ on $\Gamma$. If $k=2$, we additionally suppose $\partial_\nu (\Delta c_0) = \beta (\cdot, 0)$ on $\Gamma$. Then there is a unique solution of the PDE (\ref{pde:hoSPE}) given by $c\in H^{4+k,1+k/4}(\Omega_T)$ satisfying the bound
\begin{equation}\label{bdd:SPEwithBC} 
\|c\|_{H^{4+k,1+k/4}(\Omega_T)} \leq C(\Omega,T)\left(\|c_0\|_{H^{k+2}(\Omega)} + \|g\|_{H^{k,k/4}(\Omega_T)}  +\|\alpha\|_{H^{\mu_1,\lambda_1}(\Sigma_T)} +\|\beta\|_{H^{\mu_3,\lambda_3}(\Sigma_T)}\right) .
\end{equation}
\end{thm}

\begin{remark}
The above theorem holds for any choice of norms on the anisotropic Sobolev spaces, so long as you are willing to change the constant $C(\Omega , T)$. In applications within this paper, it will be important to control exactly how this constant depends on $T,$ so we will often extend our considerations to a domain with $T=1$, and control the dependence on $T$ by other means.
\end{remark}
\begin{remark}\label{rmk:BCcompat}
We remark that the compatibility conditions for the initial and boundary data necessarily arise due to the embedding $H^{4+k,1+k/4}(\Omega_T) \hookrightarrow C([0,T];H^{2+k}(\Omega))$ (see Theorem \ref{thm:traceSobo}), and these conditions allow one to study the PDE (\ref{pde:hoSPE}) by first lifting to homogeneous boundary conditions. Generally, Theorem \ref{thm:SPEwithBC} holds for most values of fractional $k$ given that there is a lifting of the boundary and initial condition of the appropriate order; specifically, there must exist $w \in H^{4+k,1+k/4}(\Omega_T)$ such that 
\begin{equation}\label{eqn:BCrel}
\left\{
\begin{aligned}
&w(\cdot, 0) = c_0 \quad & \text{ on }  \Omega, \\
&\partial_\nu w = \alpha \quad &\text{ on } \Sigma_T,\\
&\partial_\nu (\Delta w) = \beta  \quad & \text{ on } \Sigma_T.
\end{aligned}
\right.
\end{equation}
Understanding whether or not such a $w$ exists amounts to understanding the trace of a function in $H^{4+k,1+k/4}(\Omega_T)$ on the parabolic boundary $(\Omega\times \{0\})\cup \Sigma_T$. This study was undertaken by Grisvard in \cite{grisvard-interpolation}, and is also studied by Lions and Magenes \cite{LM-v2}. We note that, for convenience, our statement of Theorem \ref{thm:SPEwithBC} avoids $k= 1+1/2$ as this is the critical regularity for taking the trace in time of $\partial_\nu (\Delta c)$ on $\Sigma_T$. In this case, a nonlocal compatibility condition between $\beta$ and $c_0$ arises to guarantee the existence of $w$ in (\ref{eqn:BCrel}).

\end{remark}

\section{Fixed Point Argument}\label{sec:regular}
In this section, we use a fixed point theorem to prove there there is a solution in $H^{5,1+1/4}(\Omega_T)$ of a truncated CHR model (below in (\ref{pde:truncCHR})). To briefly highlight why we introduce a modified model, consider a Lipschitz function $F$ and the composition map defined by $v \mapsto F(v)$. This map is linearly bounded from $H^{1}(\Omega)$ to $H^1(\Omega),$ that is,
\begin{equation}\nonumber
\|F(v)\|_{H^1(\Omega)} \leq C_F(\|v\|_{H^1(\Omega)} +1),
\end{equation}
and we could hope the same would hold of composition operators from $H^2(\Omega)$ to $H^2(\Omega).$ But, the analogy immediately breaks. To see this, consider  
\begin{equation}\nonumber
 \Delta F (v) = F''(v)\|\nabla v\|^2 +  F'(v)\Delta v.
\end{equation}
Conveniently, the second term is linearly bounded by the $H^2(\Omega)$ norm of $v$, but the first term is quadratic, and it is impossible to avoid this nonlinearity without modification.

To sidestep some of the complications arising from composition, we introduce a truncated Cahn-Hilliard reaction model.
For $\alpha >0$, we let $\psi_\alpha \in C^\infty(\R)$ be a truncation function such that $\psi_\alpha (x) = x$ for all $x\in (-\alpha,\alpha)$, $\|\psi_\alpha'\|_\infty \leq 2$, and $\psi_\alpha' = 0$ on $(-\alpha-1,\alpha+1)^C.$ We define $\Psi_\alpha(x):=(\psi_\alpha(x_1),\ldots, \psi_\alpha (x_N))$ and the \textit{truncated Laplacian} as
\begin{equation}\label{eqn:DeltafTrunc}
 (\Delta f')_\alpha (c) := f_\alpha'''(c)\|\Psi_\alpha(\nabla c)\|^2  +  f_\alpha''(c)\psi_\alpha(\Delta c),
\end{equation}
where $f_\alpha \in C^5(\Omega),$ $f_\alpha = f$ on the interval $[1/\alpha ,1- 1/\alpha]$ and $f_\alpha '' = 0$ on $[0,1]^C$.
We define the \textit{truncated CHR model}:
\begin{equation}\label{pde:truncCHR}
\left\{\begin{aligned}
& \partial_t c  +\Delta^2 c = (\Delta f')_\alpha (c) & \text{in }\Omega_T, \\
& \partial_\nu c = 0 & \text{on }\Sigma_T, \\
& \partial_\nu (\Delta c) = \mathcal{R}(c,\Delta c) & \text{on }\Sigma_T,\\
& c(0) = c_0 & \text{in } \Omega,
\end{aligned}\right.
\end{equation} 
where 
\begin{equation}\label{ass:Rderivatives}
\mathcal{R}\in C^3(\R^2) \text{ is bi-Lipschitz with } \|\nabla^m \mathcal{R}\|_\infty <\infty \text{ for } m=1,2,3 ,
\end{equation}
 and satisfies 
\begin{equation}\label{def:Rfp}
\begin{aligned}
\mathcal{R} (s,w) &=  -R(s,-w+f'(s)) \quad  \text{ for } \quad  (s,w)\in [1/\alpha,1-1/\alpha]\times [-\alpha,\alpha].
\end{aligned}
\end{equation}
In the following, we will be interested in finding regular solutions of (\ref{pde:truncCHR}) because if $c$, $\nabla c$, and $\nabla^2 c $, in (\ref{pde:truncCHR}), are continuous, for $T$ sufficiently small and $\alpha$ well chosen, we will recover a solution to the CHR model (\ref{pde:CHRnoElastic}). To find such sufficiently regular solutions, we need the following a priori estimates for strong solutions, which will be proven in Section \ref{sec:aprioriEst}.
\begin{thm}\label{thm:CHRcont}
Let $\Omega \subset \R^N$, with $N\geq 2$, be an open, bounded set with smooth boundary. Suppose $c \in H^{4 + k,1+k/4}(\Omega_T)$ for $k =0$ or $k=1/2$ is a solution of the truncated CHR model (\ref{pde:truncCHR}) with $c_0 \in H^{2+k}(\Omega)$ satisfying $\partial_\nu c_0 = 0$. Further, let $\tilde c_0 \in H^{4+k}(\Omega)$ be such that $\partial_\nu \tilde c_0 = 0$. Then there is $T_0 = T_0 (c_0 ,\tilde c_0)$ such that for all $T<T_0$ the following estimate holds with $C>0$ independent of $T$, $c_0$, and $\tilde c_0$:
\begin{equation}\label{bdd:CHRcont}
\|c-\tilde c_0\|_{H^{4+k,1+k/4}(\Omega_T)}+ \| c-\tilde c_0\|_{L^\infty(0,T;H^{2+k}(\Omega))}\leq  C\| c_0-\tilde c_0\|_{H^{2+k}(\Omega)} + \eta(c_0, \tilde c_0,T),
\end{equation}
where $\eta(c_0, \tilde c_0,T)\to 0$ as $T \to 0.$
\end{thm}

\begin{remark}
In general $\tilde c_0$ cannot be replaced by $c_0$ in inequality (\ref{bdd:CHRcont}). Supposing the inequality did hold with this replacement, we would have $c(t)-c_0 \in H^{4+k}(\Omega)$ for $t$-a.e., and as the bound (\ref{bdd:CHRcont}) also holds for $\tilde c_0 = 0$, this would imply $c_0 \in H^{4+k}(\Omega)$, which we do not assume.
\end{remark}

In the subsequent theorem, we use a fixed point argument to prove that there is a solution of the truncated CHR model (\ref{pde:truncCHR}) belonging to $H^{5,1+1/4}(\Omega_T)$. We note that estimate (\ref{bdd:CHRcont}) can be found for $k=1/2$ from $k=0$ assuming only that the solution $c$ of the truncated CHR model belongs to $ H^{4,1}(\Omega_T)$ and performing a short bootstrap argument. As Remark \ref{rmk:bootstrapNecessity} shows, using this bootstrap approach to gain any further regularity is challenging, if not impossible! In the proof of Theorem \ref{thm:CHRreg514}, we instead turn to a fixed point argument to gain the next half-step of regularity. This allows us to start by assuming our solution has the desired regularity and only then prove the necessary estimates.

\begin{thm}\label{thm:CHRreg514}
Suppose $\Omega \subset \R^3$ is an open bounded set with smooth boundary. If $c_0\in H^3(\Omega)$ such that $\partial_\nu c_0 = 0$, then there is $T>0$ such that a solution $c$ of the truncated CHR model (\ref{pde:truncCHR}) exists on the interval $\Omega_T$ satisfying the estimate 
\begin{equation}\nonumber
\|c\|_{H^{5,1+1/4}(\Omega_T)} \leq C(f,c_0, \mathcal{R},\Omega, T).
\end{equation}
\end{thm}

\begin{proof}
We apply Schaefer's fixed point theorem to obtain existence of a solution.
Define the operator $A:v \mapsto c$ by the PDE
\begin{equation}\label{def:fixedptA}
\left\{\begin{aligned}
& \partial_t c  + \Delta^2 c =(\Delta f')_\alpha(v) & \text{in } \Omega_T, \\
& \partial_\nu c = 0 & \text{on } \Sigma_T, \\
& \partial_\nu (\Delta c) = \mathcal{R}(v,\Delta v) & \text{on } \Sigma_T, \\
& c(0) = c_0 & \text{in } \Omega .
\end{aligned}\right.
\end{equation}
We choose the domain of $A$ such that $A$ is both compact and ${\rm{range}}(A) \subset H^{5,1+1/4}(\Omega_T)$. Define the Banach space \begin{equation}\label{def:BfpSpace}
\mathcal{B} : = H^{4,1}(\Omega_T) \cap L^{2}(0,T; W^{4,r}(\Omega)) \cap L^{\infty}(0,T;W^{2,r}(\Omega))
\end{equation}
equipped with the sum of norms, for $r<6$ yet to be determined (see (\ref{eqn:estDelta}) for the choice). We \textbf{claim} 
\begin{equation}\label{eqn:compactB}
H^{5,1+1/4}(\Omega_T) \hookrightarrow \hookrightarrow \mathcal{B},
\end{equation} 
which will follow from recalling a series of embeddings. As $$H^{5,1+1/4}(\Omega_T) \hookrightarrow\hookrightarrow [H^{5,1+1/4}(\Omega_T), H^{0,0}(\Omega_T)]_\theta$$ for $\theta \in (0,1)$ (see, e.g., Exercise 16.26 of \cite{LeoniBook}, or \cite{LM-v2}), it suffices to show that 
\begin{equation}\label{eqn:embedTheta}
[H^{5,1+1/4}(\Omega_T), H^{0,0}(\Omega_T)]_\theta = H^{\theta 5,\theta (1+1/4)}(\Omega_T) \hookrightarrow \mathcal{B}
\end{equation} for some $\theta \in (0,1)$. Note by Proposition \ref{prop:anisoInterp}, $H^{\theta 5,\theta (1+1/4)}(\Omega_T) \hookrightarrow H^{4,1}(\Omega_T)$ for $\theta$ sufficiently close to $1$. Using the embedding $H^s(\Omega) \hookrightarrow L^{p^*_s}(\Omega)$ for $s \in (0,1)$ and $p^*_s : = \frac{6}{6-s2}$ (see Theorem 17.51 of \cite{LeoniBook}), for $\theta$ sufficiently close to $1$, we have $H^{\theta 5}(\Omega) \hookrightarrow W^{4,r}(\Omega)$ for given $r <6 = 2^*$. It immediately follows that $H^{\theta 5,\theta (1+1/4)}(\Omega_T) \hookrightarrow L^{2}(0,T; W^{4,r}(\Omega)).$ To conclude the last embedding necessary to prove (\ref{eqn:embedTheta}) we make use of the trace Theorem \ref{thm:traceSobo}, which shows that $c \in H^{\theta 5,\theta (1+1/4)}(\Omega_T)$ is continuously embedded in the space of ${C}([0,T]; [H^{\theta 5}(\Omega),L^2(\Omega)]_{2/(\theta 5)}).$ As $[H^{\theta 5}(\Omega),L^2(\Omega)]_{2/(\theta 5)} = H^{\theta 5-2}(\Omega)$ by Proposition \ref{prop:anisoInterp} (for $\theta>2/5$), we may once again make use of a Besov embedding theorem \cite{LeoniBook} to conclude for any $r<6 = 2^*,$ there is $\theta $ sufficiently close to $1$ such that $H^{\theta 5,\theta (1+1/4)}(\Omega_T) \hookrightarrow L^{\infty}(0,T;W^{2,r}(\Omega)).$ This concludes the claim.

We now prove that the hypotheses of Schaefer's fixed point theorem \cite{EvansPDE} are satisfied by the operator $A:\mathcal{B} \to \mathcal{B}$ for fixed initial data $c_0$. These hypotheses are characterized as
\begin{itemize}
\item \text{Compactness:} The functional ${A}:\mathcal{B} \to \mathcal{B}$ is compact.
\item Continuity: The functional ${A}:\mathcal{B} \to \mathcal{B}$ is continuous.
\item Boundedness: The set $\{c\in \mathcal{B}:\zeta A[c] = c, \zeta \in (0,1])\}$ is bounded in the norm of $\mathcal{B}$.
\end{itemize}
Supposing the claim, by Schaefer's fixed point theorem, we find there is a $c\in \mathcal{B}$ such that $A [c] = c,$ and by the argument for boundedness below, $c \in H^{5,1+1/4}(\Omega_T)$ thereby finishing the theorem. So, it only remains to verify the hypotheses are satisfied. We remark that satisfying the first two hypotheses is necessary housekeeping, and the proof of boundedness contains the main novelty of our argument.

 \vspace{0.5em}
 
\noindent \textbf{Compactness.} \sloppy We estimate $\mathcal{R}(v,\Delta v)$ and $(\Delta f')_\alpha(v)$ in the norms for $H^{3/2,3/8}(\Sigma_T)$ and $H^{1,1/4}(\Omega_T)$ respectively. We define $\beta_v := \mathcal{R}(v,\Delta v) .$ 
As $\mathcal{R}$ is Lipschitz (see (\ref{ass:Rderivatives})), we can use Proposition \ref{prop:intermediateReg} to conclude 
\begin{equation}\label{bdd:betaC}
\|\beta_v\|_{H^{1,1/2}(\Omega_T)}\leq C(\Omega,T)(\|v\|_{H^{4,1}(\Omega_T)} +1).
\end{equation}
By (\ref{bdd:betaC}) and continuity of the trace map (see Theorem \ref{thm:LMnormalReg}), we have
\begin{equation}\label{est:R514}
\begin{aligned}
\|\beta_v\|_{H^{3/2,3/8}(\Sigma_T)} & \leq C(\Omega, T) \|\beta_v\|_{H^{2,1/2}(\Omega_T)} \\
& \leq C(\mathcal{R},\Omega, T)( \|\nabla^2 \beta_v\|_{H^{0,0}(\Omega_T)} + \|v\|_{H^{4,1}(\Omega_T)}+1).
\end{aligned}
\end{equation}
For simplicity, we show how to control the higher order spatial derivatives of $\beta_v$ by looking at two derivatives in the same direction (mixed derivatives are similar):
\begin{equation}\label{eqn:RsecondOrder}
\begin{aligned}
\partial_i^2 \beta_v = &  (\partial_s^2 \mathcal{R}(v,\Delta v)(\partial_i v)^2 + 2(\partial_s\partial_w \mathcal{R})(v,\Delta v)\partial_i v\partial_i \Delta v +(\partial_s \mathcal{R})(v,\Delta v)\partial_i^2 v  \\
& + (\partial_w^2 \mathcal{R})(v,\Delta v)(\partial_i \Delta v)^2 + (\partial_w \mathcal{R})(v,\Delta v)\partial_i^2 \Delta v. \\
\end{aligned}
\end{equation}
As $\mathcal{R}$ is Lipschitz with bounded derivatives (see (\ref{ass:Rderivatives})), up to a constant $C(\mathcal{R})$, Young's inequality implies that $\partial_i^2 \beta_v$ is controlled in $H^{0,0}(\Omega_T)$ by $\|v\|_{H^{4,1}(\Omega_T)}$, $\|(\partial_i v)^2\|_{H^{0,0}(\Omega_T)}$, and  $\|(\partial_i \Delta v)^2\|_{H^{0,0}(\Omega_T)}$; the norm for $(\partial_i v)^2$ may be controlled using a simplified version of the argument to bound the last term. To control the last term, we apply the Gagliardo-Nirenberg inequality (Theorem \ref{thm:nirenbergIneq}) for $t$-a.e. to find
\begin{equation}\label{est:DeltaNiren}
\|\partial_i \Delta v\|_{L^4(\Omega)} \leq C(\|\nabla^2 \Delta v\|_{L^r(\Omega)}^a \|\Delta v\|_{L^r(\Omega)}^{1-a} + \|\Delta v\|_{L^r(\Omega)}),
\end{equation} 
where  $a= \frac{3}{2}(\frac{1}{12}+\frac{1}{r}).$ Choosing $4\leq r<6$ (equivalently, $\frac{3}{8}<a\leq 1/2$), we have 
\begin{equation}\label{eqn:estDelta}
\begin{aligned}
\|(\partial_i \Delta v)^2\|_{H^{0,0}(\Omega_T)}  = & \left(\int _0^T\|\partial_i \Delta v\|_{L^4(\Omega)}^4 \, dt\right)^{1/2} \\
\leq &C \|\Delta v\|_{L^\infty (0,T; L^r(\Omega))}^{2(1-a)} \left(\int _0^T\|\nabla^2 \Delta v \|_{L^r(\Omega)}^{4a} \, dt \right)^{1/2} + C\|\Delta v\|_{L^\infty (0,T; L^r(\Omega))}^2 \\
\leq & C\|\Delta v\|_{L^\infty (0,T; L^r(\Omega))}^{2(1-a)} \left(\int _0^T\|\nabla^2 \Delta v \|_{L^r(\Omega)}^{2}\, dt \right)^{1/2} + C\|\Delta v\|_{L^\infty (0,T; L^r(\Omega))}^{2(1-a)+1} +C \\
\leq & C\|v\|_{\mathcal{B}}^{2(1-a)+1} +C,
\end{aligned}
\end{equation} 
where we used the inequality $s^2 \leq C(s^{2(1-a)+1}+1)$ and H\"older's inequality since $a\leq 1/2$.
Thus, using the definition of the space $\mathcal{B},$ (\ref{est:R514}), and (\ref{eqn:estDelta}), we have
\begin{equation}
\nonumber 
\|\beta_v\|_{H^{3/2,3/8}(\Sigma_T)} \leq C\|v\|_{\mathcal{B}}^{2(1-a)+1} + C.
\end{equation}
Due to the truncated Laplacian (\ref{eqn:DeltafTrunc}), estimation of the bulk term $(\Delta f')_\alpha(v)$ in $H^{1,1/4}(\Omega_T)$ is straightforward. By Theorem \ref{thm:SPEwithBC} (with $k=1$), we conclude
\begin{equation}\nonumber
\|c\|_{H^{5,1+1/4}(\Omega_T)}\leq C\|v\|_{\mathcal{B}}^{2(1-a)+1} + C, 
\end{equation}
which implies $A:\mathcal{B}\to \mathcal{B}$ is compact by (\ref{eqn:compactB}).

 \vspace{0.5em}

\noindent \textbf{Continuity.} Suppose $v_n \to v$ in $\mathcal{B}$. To show that $A[v_n]\to {A}[v]$, by Theorem \ref{thm:SPEwithBC} (with $k=1$) and the claim preceding (\ref{eqn:embedTheta}), it is sufficient to show that the data converges as follows:
\begin{equation} \nonumber
\begin{aligned}
(\Delta f')_\alpha(v_n) & \to (\Delta f')_\alpha(v) & \text{ in }H^{1,1/4}(\Omega_T), \\
\beta_{v_n}:=\mathcal{R}(v_n,\Delta v_n) & \to \mathcal{R}(v,\Delta v) = \beta_v & \text{ in }H^{2,1/2}(\Omega_T),
\end{aligned}
\end{equation}
where we have used Theorem \ref{thm:LMnormalReg} to reduce our consideration to convergence on $\Omega_T$ versus $\Sigma_T.$ We focus our attention on the second convergence, the first being simpler. 

Up to a subsequence, we may assume $\nabla^m v_n \to \nabla^m v$ a.e. in $\Omega_T$ for $m\in \{0,\ldots, 4\}.$ To see that $\beta_{v_n} \to \beta_v$ in $H^{0,0}(\Omega_T),$ recall $\mathcal{R}$ is Lipschitz (see (\ref{ass:Rderivatives})) to find
\begin{equation}\nonumber
\begin{aligned}
\|\beta_{v_n} - \beta_v\|_{H^{0,0}(\Omega_T)}^2  =& \int_0^T \int_\Omega |\mathcal{R}(v_n,\Delta v_n) - \mathcal{R}(v,\Delta v)|^2 \ dx \ dt \\
\leq & C(\mathcal{R})\int_0^T \int_\Omega \left( |v_n - v|^2 + |\Delta v_n -\Delta v|^2 \right) \ dx \ dt.
\end{aligned}
\end{equation}
Thus by the convergence of $v_n$ in $\mathcal{B}$, we directly have convergence in $H^{0,0}(\Omega_T).$
To prove convergence in $H^{0,1/2}(\Omega_T)$ we argue using (\ref{ass:Rderivatives}) and the Gagliardo type semi-norm (see (\ref{def:besovAnisoNorm})):
\begin{equation}\nonumber
\begin{aligned}
|\beta_{v_n} &  - \beta_v|_{H^{0,1/2}(\Omega_T)}^2  \\
=  & \int_\Omega  \int_0^T \int_0^T \frac{|\mathcal{R}(v_n,\Delta v_n)(t) - \mathcal{R}(v,\Delta v)(s)|^2}{|t-s|^{2}} \ dt \ ds \ dx \\
\leq &  C(\mathcal{R}) \int_\Omega  \int_0^T \int_0^T \frac{|v_n(t) - v(s)|^2}{|t-s|^{2}} + \frac{|\Delta v_n(t) - \Delta v(s)|^2}{|t-s|^{2}}  \ dt \ ds \ dx \\
= & C(\mathcal{R}) \left(|v_n - v|_{H^{0,1/2}(\Omega_T)}^2  +|\Delta v_n - \Delta v|_{H^{0,1/2}(\Omega_T)}^2\right).
\end{aligned}
\end{equation}
As $(v_n,\Delta v_i )\to (v,\Delta v )$ in $[H^{2,1/2}(\Omega_T)]^2$ by Proposition \ref{prop:intermediateReg}, we are done. Convergence of first order derivatives in space is done similarly. 

To show that the second order derivatives converge is more involved. We show convergence for repeated derivatives as in (\ref{eqn:RsecondOrder}), with mixed derivatives being similar. We explicitly show convergence of the term $(\partial_w^2 \mathcal{R})(v_n,\Delta v_n)(\partial_i \Delta v_n)^2$ with the remaining terms being simpler. Decomposing the difference of products, for $t-$a.e. we compute
\begin{equation}\label{est:breakdown1}
\begin{aligned}
 \|(\partial_w^2 \mathcal{R})(v_n,\Delta v_n)(\partial_i \Delta v_n)^2 & -(\partial_w^2 \mathcal{R})(v,\Delta v)(\partial_i \Delta v)^2\|_{L^2(\Omega)}  \\
\leq & \|(\partial_w^2 \mathcal{R})(v_n,\Delta v_n)\left[(\partial_i \Delta v_n)^2 -(\partial_i \Delta v)^2\right]\|_{L^2(\Omega)}   \\
& + \|\left[(\partial_w^2 \mathcal{R})(v_n,\Delta v_n) -(\partial_w^2 \mathcal{R})(v,\Delta v)\right](\partial_i \Delta v)^2\|_{L^2(\Omega)} \\
\leq & C(\mathcal{R}) \|(\partial_i \Delta v_n)^2 -(\partial_i \Delta v)^2\|_{L^2(\Omega)} \\
& + \|\left[(\partial_w^2 \mathcal{R})(v_n,\Delta v_n) -(\partial_w^2 \mathcal{R})(v,\Delta v)\right](\partial_i \Delta v)^2\|_{L^2(\Omega)} ,
\end{aligned}
\end{equation}
where we used that derivatives of $\mathcal{R}$ are bounded from (\ref{ass:Rderivatives}).
Up to another subsequence of $n,$ for $t$-a.e., the second term goes to $0$ by the Lebesgue dominated convergence theorem. Taking another subsequence if necessary, we apply H\"older's inequality and the Sobolev-Gagliardo-Nirenberg embedding theorem to show that the first term also goes to $0$ for $t$-a.e.:
\begin{equation}\label{eqn:splitConverge1}
\begin{aligned}
\|(\partial_i \Delta v_n)^2 -(\partial_i \Delta v)^2\|_{L^2(\Omega)} \leq &\, \|\partial_i \Delta v_n -\partial_i \Delta v\|_{L^4(\Omega)} \|\partial_i \Delta v_n +\partial_i \Delta v\|_{L^4(\Omega)} \\
\leq &\, C_\Omega \|\nabla \Delta (v_n - v)\|_{H^1(\Omega)}\|\nabla \Delta (v_n + v)\|_{H^1(\Omega)} \\
\to & \, 0 \cdot 2\|\nabla \Delta v\|_{H^1(\Omega)}  =0.
\end{aligned}
\end{equation}
We now apply the generalized Lebesgue dominated convergence theorem to prove $(\partial_i \Delta v_n)^2 \to(\partial_i \Delta v)^2$ in $H^{0,0}(\Omega_T)$. Writing
\begin{equation}\label{def:bulkSecond}
\begin{aligned}
\|(\partial_i \Delta v_n)^2 -(\partial_i \Delta v)^2\|_{H^{0,0}(\Omega_T)}^2 = & \int_0^T \|(\partial_i \Delta v_n)^2 -(\partial_i \Delta v)^2\|^2_{L^2(\Omega)} \ dt,
\end{aligned}
\end{equation}
we bound the integrand pointwise for $t$-a.e. using estimate (\ref{est:DeltaNiren}) and that $\|v_n\|_{\mathcal{B}} \to \|v\|_{\mathcal{B}}$ in $\mathcal{B}$:
\begin{equation}\label{eqn:splitConverge2}
\begin{aligned}
\|(\partial_i& \Delta v_n)^2 -(\partial_i \Delta v)^2\|^2_{L^2(\Omega)}  \\
\leq & C\left(\|\partial_i \Delta v_n\|_{L^4(\Omega)}^4 +\|\partial_i \Delta v\|^4_{L^4(\Omega)} \right) \\
\leq &C\left(\| \Delta v_n \|_{W^{2,r}(\Omega)}^{4a} \|\Delta v_n\|_{L^r(\Omega)}^{4(1-a)} + \| \Delta v \|_{W^{2,r}(\Omega)}^{4a} \|\Delta v\|_{L^r(\Omega)}^{4(1-a)}\right) \\
\leq &C\sup_k \left\{\|\Delta v_k\|_{L^\infty(0,T;L^r(\Omega))}^{4(1-a)}\right\}\left(\| \Delta v_n \|_{W^{2,r}(\Omega)}^{2}  + \| \Delta v \|_{W^{2,r}(\Omega)}^{2} +1\right)  \\
\leq &C\sup_k \left\{\|v_k\|_{\mathcal{B}}^{4(1-a)}\right\}\left(\| \Delta v_n \|_{W^{2,r}(\Omega)}^{2}  + \| \Delta v \|_{W^{2,r}(\Omega)}^{2} +1\right) \in L^1(0,T).
\end{aligned}
\end{equation}
Thus using (\ref{eqn:splitConverge1}) and (\ref{eqn:splitConverge2}), we apply the generalized Lebesgue dominated convergence theorem to conclude (\ref{def:bulkSecond}) converges to $0.$ Likewise, we conclude that the left-hand side of (\ref{est:breakdown1}) goes to $0$ in $L^2(0,T)$, from which we conclude the desired convergence of second order terms, and finally continuity of the operator $A.$

 \vspace{0.5em}

\noindent \textbf{Boundedness.} We show that the set 
\begin{equation}\label{def:boundingSet}
\{c \in \mathcal{B}: c = \zeta A[c] \text{ for }\zeta \in (0,1]\}
\end{equation}
is bounded in $H^{5,1+1/4}(\Omega_T)$ and thereby $\mathcal{B}$. We assume $\zeta=1$; the argument is the same for other $\zeta.$ Thus, suppose $c = A[c] \in H^{5,1+1/4}(\Omega_T).$
Making use of the assumptions on $f_\alpha$, it straightforward to show that $(\Delta f')_\alpha (c)$ is bounded in $H^{1,1/4}(\Omega_T)$ in terms of $\|c\|_{H^{4,1}(\Omega_T)}\leq C(c_0,\Omega, T)$, where the constant follows from Theorem \ref{thm:CHRcont}. Now, we control $\beta_c := \mathcal{R}(c,\Delta c)$ in $H^{3/2,3/8}(\Sigma_T).$ Given Proposition \ref{prop:intermediateReg}, the norm of $\beta_c$ in $H^{0,3/8}(\Sigma_T)$ is controlled by $\|c\|_{H^{4,1}(\Omega_T)}\leq C(c_0,\Omega, T)$. We summarize these initial bounds as 
\begin{equation}\label{est:data5141}
\|\beta_c\|_{H^{0,3/8}(\Sigma_T)} + \|(\Delta f')_\alpha (c)\|_{H^{1,1/4}(\Omega_T)} \leq C(c_0,f, \mathcal{R},\Omega, T). \\
\end{equation}
 To bound $\beta_c$ in $H^{3/2,0}(\Sigma_T)$, we look at $\nabla^2 \beta_c$ in $H^{0,0}(\Omega_T).$ We control the repeated derivative $\partial_i^2 \beta_c,$  as in (\ref{eqn:RsecondOrder}); control of the mixed derivatives is analogous.

First we need an estimate which follows from Theorem \ref{thm:CHRcont}. Introducing $\tilde c_0 \in H^{4+1/2}(\Omega)$ with $\partial_\nu \tilde c_0 = 0$, by (\ref{bdd:CHRcont}) and noting the embedding $H^{1/2}(\Omega) \hookrightarrow L^3(\Omega) = L^{2^*_{1/2}}(\Omega)$ \cite{LeoniBook}, for all $T>0$ sufficiently small (depending on $c_0$ and $\tilde c_0$), we have 
\begin{equation}\label{bdd:DeltacLinfty}
\|\nabla^2 (c-\tilde c_0)\|_{L^\infty(0,T;L^3(\Omega))} \leq C\|c_0-\tilde c_0 \|_{H^{2+1/2}(\Omega)} + \eta(c_0,\tilde c_0, T) =: \lambda +\eta ,
\end{equation}
where the constant $C>0$ is independent of $c_0,$ $\tilde c_0$, and $T.$
Looking to ultimately bound terms arising in (\ref{eqn:RsecondOrder}), we square the Gagliardo-Nirenberg inequality (Theorem \ref{thm:nirenbergIneq}) in dimension $N=3$ and use (\ref{bdd:DeltacLinfty}) to find for $t$-a.e.
\begin{equation}\label{bdd:Deltac2}
\begin{aligned}
\|\partial_i (\Delta c)\|_{L^4(\Omega)}^2\ \leq & C\left(\|\partial_i (\Delta (c-\tilde c_0))\|_{L^4(\Omega)}^2\ + \|\partial_i (\Delta \tilde c_0)\|_{L^4(\Omega)}^2\right)\\\
\leq & C(\Omega)\left(\|\Delta (c-\tilde c_0)\|_{H^3(\Omega)}\|\Delta (c-\tilde c_0)\|_{L^3(\Omega)} + \|\partial_i (\Delta \tilde c_0)\|_{L^4(\Omega)}^2 \right)\\
 \leq & C(\Omega)\left((\lambda +\eta)\|\Delta (c-\tilde c_0)\|_{H^3(\Omega)} + \|\partial_i (\Delta \tilde c_0)\|_{L^4(\Omega)}^2\right) \\
 \leq &C(\Omega)\left( (\lambda + \eta)\|\Delta c\|_{H^3(\Omega)} + (\lambda + \eta)\|\Delta \tilde c_0\|_{H^3(\Omega)} + \|\partial_i (\Delta \tilde c_0)\|_{L^4(\Omega)}^2\right).
\end{aligned}
\end{equation}
Using H\"older's inequality, Young's inequality, the Sobolev-Gagliardo-Nirenberg embedding theorem (which shows that $\|\partial_i c\|_{L^4(\Omega)}\leq C\|c\|_{H^2(\Omega)}$), and the trace Theorem \ref{thm:traceSobo}, we have for $t$-a.e. 
\begin{equation}\label{bdd:prodpartialc}
\begin{aligned}
\|\partial_i c\, \partial_i (\Delta c)\|_{L^2(\Omega)} \leq & C(\Omega)\|c\|_{L^\infty(0,T;H^2(\Omega))}^2  + \|\partial_i (\Delta c)\|_{L^4(\Omega)}^2  \\
\leq & C(\Omega,T)\|c\|_{H^{4,1}(\Omega_T)}^2 +\|\partial_i (\Delta c)\|_{L^4(\Omega)}^2.
\end{aligned}
\end{equation}
Recalling that $\mathcal{R}$ has bounded derivatives (see (\ref{ass:Rderivatives})) and noting the terms in (\ref{eqn:RsecondOrder}), we see that the bounds (\ref{bdd:Deltac2}), (\ref{bdd:prodpartialc}), and $\|c\|_{H^{4,1}(\Omega_T)}\leq C$  imply
\begin{equation}\label{bdd:partial2wc}
\begin{aligned}
\|\partial_i^2 \beta_c \|_{H^{0,0}(\Omega_T)} \leq  & C(c_0, f, \mathcal{R},\Omega)\left(C(T) + \left(\lambda+\eta\right)\|\Delta c\|_{H^{3,0}(\Omega_T)}\right).
\end{aligned}
\end{equation}
As noted previously, an argument analogous to the above succeeds in controlling the full derivative $\nabla^2 \beta_c$. Thus by (\ref{bdd:betaC}), (\ref{bdd:partial2wc}), and the trace inequality $\|\cdot\|_{H^{3/2}(\Gamma)}\leq C(\Omega)\|\cdot\|_{H^2(\Omega)}$ \cite{LeoniBook}, we have 
\begin{equation}\label{est:data5142}
\begin{aligned}
&\|\beta_c\|_{H^{3/2,0}(\Sigma_T)} \leq C(c_0, f, \mathcal{R},\Omega)\left(C(T) + \left(\lambda+\eta\right)\|\Delta c\|_{H^{3,0}(\Omega_T)}\right).
\end{aligned}
\end{equation}
Lastly, we will extend the bulk and boundary data to $\Omega_1$ and $\Sigma_1$, and then apply Theorem \ref{thm:SPEwithBC} (with $k=1$) to bound $c\in H^{5,1+1/4}(\Omega_T)$ by 
\begin{equation}\label{est:c514penult}
\|c\|_{H^{5,1+1/4}(\Omega_T)} \leq C(c_0,f, \mathcal{R},\Omega, T) + C(c_0,f, \mathcal{R},\Omega)\left( \lambda +\eta \right)\|\Delta c\|_{H^{3,0}(\Omega_T)}.
\end{equation}
Supposing we have estimate (\ref{est:c514penult}), we may conclude the theorem as follows. First, we show we can choose $\tilde c_0 \in H^5(\Omega)$ with $\partial_\nu \tilde c_0 = 0$ such that $\lambda:= C\|c_0-\tilde c_0 \|_{H^{2+1/2}(\Omega)}$ is arbitrarily small. By the divergence theorem, $\phi_0 := \Delta c_0$ satisfies $\int_\Omega \phi_0 \, dx = 0.$ By density of smooth functions in $H^1(\Omega)$, we can choose $\phi_n \in C^\infty(\bar \Omega)$, with $\int_\Omega \phi_n\, dx = 0$, converging to $\phi_0$ in $H^1(\Omega).$ Define $c_n\in C^\infty(\bar \Omega)$ as the solution of the PDE
\begin{equation}
\left\{\begin{aligned}
& \Delta c_n  = \phi_n & \text{in } \Omega, \\
& \partial_\nu c_n = 0 & \text{on } \Gamma.
\end{aligned}\right.
\end{equation}
By elliptic regularity (see, e.g., \cite{EvansPDE}), $c_n \to c$ in $H^3(\Omega)$, and choosing $\tilde c_0 : =c_n$ for sufficiently large $n$, $\lambda$ is as small as desired. Let $\tilde c_0$ be chosen so that $\lambda   < \frac{1}{2C(c_0,f, \mathcal{R},\Omega)}$, where the constant in the denominator refers to the second constant in (\ref{est:c514penult}). Then by Theorem \ref{thm:CHRcont}, we can choose $T>0$ sufficiently small such that $\eta < \frac{1}{2C(c_0,f, \mathcal{R},\Omega)}.$
Finally, applying estimate (\ref{est:c514penult}), we may directly conclude the proof of boundedness. It only remains to prove (\ref{est:c514penult}).

We use Corollary \ref{cor:fracExtension} and (\ref{def:besovAnisoNorm}) to find an extension $\tilde{\beta_c} \in H^{3/2,3/8}(\Omega_1)$ of $\beta_c$ such that
\begin{equation}\nonumber
\|\tilde{\beta_c}\|_{H^{3/2,3/8}(\Sigma_1)} \leq C\left((1+T^{-3/8})\|\beta_c\|_{H^{0,0}(\Sigma_T)} + \|\beta_c\|_{H^{3/2,3/8}(\Sigma_T)}\right).
\end{equation}
Likewise, we find an extension $\tilde{f} \in H^{1,1/4}(\Omega_1)$ of $(\Delta f')_\alpha (c)$ such that
\begin{equation}\nonumber
\|\tilde{f}\|_{H^{1,1/4}(\Omega_1)} \leq C\left((1+T^{-1/4})\|(\Delta f')_\alpha (c)\|_{H^{0,0}(\Omega_T)} + \|(\Delta f')_\alpha (c)\|_{H^{1,1/4}(\Omega_T)}\right).
\end{equation}
With this in hand, we consider the PDE for $\bar{c}$:
\begin{equation}\nonumber
\left\{\begin{aligned}
& \partial_t \bar{c} + \Delta^2 \bar{c} = \tilde{f}  & \text{in } \Omega_1,\\
& \partial_\nu \bar{c} = 0 & \text{on } \Sigma_1,\\
& \partial_\nu (\Delta \bar{c}) =\tilde{\beta_c} & \text{on } \Sigma_1,\\ 
& \bar{c}(0) = c_0 & \text{in } \Omega. 
\end{aligned}\right.
\end{equation}
Note, by uniqueness (see Theorem \ref{thm:SPEwithBC}), $\bar{c}|_{\Omega_T} = c.$ 
Theorem \ref{thm:SPEwithBC} (with $k=1$), bound (\ref{est:data5141}), and (\ref{est:data5142}) then show
\begin{equation}\nonumber
\begin{aligned}
\|c\|_{H^{5,1+1/4}(\Omega_T)}\leq  \|\bar{c}\|_{H^{5,1+1/4}(\Omega_1)} \leq & C(\Omega,1)(\|\tilde{f}\|_{H^{1,1/4}(\Omega_1)} +\|\tilde{\beta_c}\|_{H^{3/2,3/8}(\Sigma_1)} ) \\
\leq & C(c_0,f, \mathcal{R},\Omega, T) + C(c_0,f, \mathcal{R},\Omega)\left( \lambda +\eta \right)\|\Delta c\|_{H^{3,0}(\Omega_T)}.
\end{aligned}
\end{equation}
Note the first constant may blow up as $T \to 0,$ and the extensions in time have been used to guarantee this does not happen to the coefficient of $\lambda +\eta.$

\end{proof}

\begin{remark}\label{rmk:bootstrapNecessity}
We discuss why application of a fixed point theorem appears necessary in the previous proof. To prove Theorem \ref{thm:CHRreg514} by directly bootstrapping, we would need to show that for a solution $c \in H^{4+1/2,1+1/8}(\Omega_T)$, the boundary data defined by $\mathcal{R}(c,\Delta c)$ belongs to $H^{3/2,3/8}(\Sigma_T)$. This would provide the necessary regularity to use Theorem \ref{thm:SPEwithBC} and conclude $c \in H^{5,1+1/4}(\Omega_T)$. Setting our sights slightly lower, we could aim to show $\mathcal{R}(c,\Delta c)\in H^{1+r,1/4+r/4}(\Sigma_T)$ for some $0<r\leq 1/2$. 

Using a covering and flattening argument, this inclusion can be reduced to the following question (for smooth $\mathcal{R}$): For $\Omega \subset \R^2$ and $S$ smooth with bounded derivatives, does $v \in H^{2,1/2}(\Omega_T)$ imply $S(v) \in H^{1+r,1/4+r/4}(\Omega_T)$? In this implication, $\Omega$ plays the role of the boundary and $v$ replaces $\Delta c$ in the original problem. We non-rigorously use $\partial_x^r$ to represent a spatial derivative of (fractional) order $r$. It is direct to conclude $S(v) \in H^{1,1/2}(\Omega_T)$ as $S$ is Lipschitz, and the difficult step is to prove $S'(v)\nabla v \in H^{r,0}(\Omega_T)$. Formally using the fractional Leibniz rule in $\Omega$ \cite{grafakos-KatoPonceIneq} gives that 
\begin{equation}\label{bdd:fracLeibniz}
\|\partial_x^r (S'(v)\nabla v)\|_{L^2(\Omega)} \leq C\left(\|S'(v)\|_{L^\infty(\Omega)} \|\partial_x^r\nabla v\|_{L^2(\Omega)} + \|\partial_x^r S'(v)\|_{L^p(\Omega)} \|\nabla v\|_{L^q (\Omega)}\right),
\end{equation}
where $\frac{1}{p} + \frac{1}{q} = \frac{1}{2}.$ The first term of the right hand side is bounded in $L^2(0,T)$ as $|S'|\leq C<\infty$. To bound the second term, we use the subcritical Sobolev embedding $\|\cdot\|_{L^{p}(\Omega)}\leq C\| \cdot \|_{W^{1-r,\tilde q}(\Omega)}$ with $p = \frac{2\tilde q}{2-(1-r)\tilde q} $ (see \cite{LeoniBook}).  Consequently, the second term is bounded above by
\begin{equation}\nonumber
C(S)\|\partial_x^r v \|_{L^p(\Omega)} \|\nabla v\|_{L^q (\Omega)} \leq C\|\partial_x^r v \|_{W^{1-r,\tilde q}(\Omega)}\|\nabla v\|_{L^q (\Omega)} \leq C\| v \|_{W^{1,\tilde q}(\Omega)}\|v\|_{W^{1,q} (\Omega)}.
\end{equation}
Using the relation between $\tilde q$, $q,$ and $p,$ we heuristically balance our energy by letting $\tilde q = q$ with $q = \frac{4}{2-r}$, which tends to $2$ from above as $r$ goes to $0.$ Combining the above equations, we have 
\begin{equation}\label{eqn:rmkbdd1}
\|S(v)\|_{H^{1+r,0}(\Omega_T)}^2 \leq C(S,\|v\|_{H^{2,1/2}(\Omega_T)}) +C\int_0^T  \|v\|_{W^{1,q} (\Omega)}^4 \, dt.
\end{equation}
As $v \in H^{2,1/2}(\Omega_T)$, $\nabla v \in H^{1,1/4}(\Omega_T)$ by Proposition \ref{prop:intermediateReg}. Noting the embedding of $$H^{1/4}(0,T;L^2(\Omega)) \hookrightarrow L^4(0,T;L^2(\Omega))$$ (which follows from the embedding $H^{1/4}(0,T) \hookrightarrow L^4(0,T)$ \cite{LeoniBook} and norm (\ref{def:besovAnisoNorm})), the best we can naively expect is 
\begin{equation}\label{eqn:rmkbdd2}
\int_0^T  \|v\|_{H^1(\Omega)}^4 <\infty.
\end{equation}
As $q>2$ for any $r>0,$ this falls short of the means to bound (\ref{eqn:rmkbdd1}) using (\ref{eqn:rmkbdd2}). 
\end{remark}

\begin{remark}
We make some comments on the fixed point approach used in the above proof. First, our choice of $\mathcal{B}$ in (\ref{def:BfpSpace}) explicitly highlights spatial regularity in terms of integer-ordered Sobolev spaces, but is essentially the same as choosing $\mathcal{B}:=H^{\theta 5, \theta (1+1/4)}(\Omega_T)$ with $0<\theta <1$ sufficiently close to $1$. For sake of discussion, supposing we have chosen $$\mathcal{B} := H^{\theta 5, \theta (1+1/4)}(\Omega_T) = H^{4+r,1+r/4}(\Omega_T),$$ for some $1/2<r<1,$ it is worth asking if we have lost any power in the estimates of \emph{Boundedness} by bounding the set (\ref{def:boundingSet}) in $H^{5,1+1/4}(\Omega_T)$ versus in $\mathcal{B}$. To see that our approach is, in some sense, equivalent, we present a formal argument to obtain the analogue of (\ref{est:c514penult}) in $\mathcal{B}$. This argument shows that the estimate from Theorem \ref{thm:CHRcont} is critical for our proof regardless of the choice of space $\mathcal{B}$.

Restricting our attention to the lifted boundary terms, we estimate the data $\mathcal{R}(c,\Delta c)$ in the correct space given by Theorem \ref{thm:SPEwithBC} for $k = r$ (which holds without modification for this range of $k$, see \cite{LM-v2}). Letting $v$ play the role of $\Delta c$ and $S \in C^\infty(\R)$ be a smooth function with bounded derivatives taking the place of $\mathcal{R}$, the bound needed is  
\begin{equation}\label{bdd:BspaceFP}
\|S(v)\|_{H^{1+r,1/4+r/4}(\Omega_T)} \leq \delta \|v\|_{H^{2+r,1/2+r/4}(\Omega_T)}+C
\end{equation}
for some $0<\delta <1$ and $\Omega\subset \R^3.$ As in Remark \ref{rmk:bootstrapNecessity}, we let $\partial_x^r$ represent a spatial derivative of (fractional) order $r$, and the challenge is to control $S'(v)\nabla v$ in $H^{r,0}(\Omega_T)$ by the right hand side of (\ref{bdd:BspaceFP}). Using the fractional Leibniz rule, (\ref{bdd:fracLeibniz}), we reduce this to bounding $\|\partial_x^r v\|_{L^p(\Omega)}\|\nabla v\|_{L^q(\Omega)}$, with $\frac{1}{p}+\frac{1}{q} = \frac{1}{2}$, in $L^2(0,T).$ Using embeddings for fractional Sobolev spaces \cite{LeoniBook}, we have $$\|\partial_x^r v\|_{L^p(\Omega)} \|\nabla v\|_{L^q(\Omega)} \leq C\|\partial_x^r v\|_{W^{1-r,\tilde q}(\Omega)} \|\nabla v\|_{L^q(\Omega)}\leq C\| v\|_{W^{1,\tilde q}(\Omega)}\|\nabla v\|_{L^q(\Omega)},$$ where $\tilde q$ satisfies $\frac{3\tilde q}{3 - (1-r)\tilde q} = p.$ Dropping lower order terms and balancing the energy by letting $q = \tilde q = \frac{12}{5-2r},$ we need to show 
\begin{equation}\label{bdd:BspaceFP2}
\left( \int_0^T\|\nabla v\|_{L^q(\Omega)}^4 \, dt \right)^{1/2} \leq \delta \|v\|_{H^{2+r,1/2+r/4}(\Omega_T)}+C
\end{equation} (note for $r=1$, we have $q=4$, recovering the left hand side of (\ref{bdd:Deltac2})). We use the Gagliardo-Nirenberg inequality (Theorem \ref{thm:nirenbergIneq}) in dimension $N=3$ to find
\begin{equation}\label{bdd:whyFP}
\|\nabla v\|_{L^q(\Omega)}^2 \leq C\left[\|v\|_{L^3(\Omega)}^{1-a}\|\nabla^2 v\|^a_{W^{2,z}(\Omega)} \right]^2 \leq \left[\|v\|_{L^3(\Omega)}^{1-a}\| v\|^a_{H^{2+r}(\Omega)} \right]^2,
\end{equation}
where $z = \frac{6}{3-2r}$ is the critical exponent for the Sobolev-Gagliardo-Nirenberg embedding theorem in the fractional Sobolev space $H^r(\Omega)$ \cite{LeoniBook}. Note by Theorem \ref{thm:CHRcont}, our best a priori bounds are on $v  \in L^\infty (0,T;H^{1/2}(\Omega)) \hookrightarrow L^\infty (0,T;L^3(\Omega))$, hence why $\|v\|_{L^3(\Omega)}$ is included in (\ref{bdd:whyFP}). With $q$ and $z$ given in terms of $r$, a computation tracking fractions shows $a$ in the Gagliardo-Nirenberg inequality (Theorem \ref{thm:nirenbergIneq}) is given by $a = 1/2$ for every choice of $1/2<r<1$. Consequently, for any $r$ in this range,
\begin{equation}\nonumber
\left(\int_0^T\|\nabla v\|_{L^q(\Omega)}^4\, dt\right)^{1/2} \leq C\|v\|_{L^\infty(0,T;L^3(\Omega))} \| v\|_{H^{2+r,0}(\Omega_T)}.
\end{equation}
It follows that to obtain the estimate (\ref{bdd:BspaceFP2}), we require the smallness of $\|v\|_{L^\infty(0,T;L^3(\Omega))}$ as in (\ref{bdd:CHRcont}) in Theorem \ref{thm:CHRcont}.

\end{remark}

\section{Proof of Theorem \ref{thm:CHRexist3D}}\label{sec:regular2}

We now prove the final step in regularity, inclusion of the solution in $H^{6,1+1/2}(\Omega_T).$ Note, in dimension $N=2$, we are directly able to bootstrap from $H^{5,1+1/4}(\Omega_T)$ to $H^{6,1+1/2}(\Omega_T)$ due to the strength of the Gagliardo-Nirenberg inequality (Theorem \ref{thm:nirenbergIneq}). However, for dimension $N=3$, we must first take an intermediate step and show that our solution almost belongs to $H^{5+1/2,1+3/8}(\Omega_T).$ 

\begin{proof}[Proof of Theorem \ref{thm:CHRexist3D}]
We complete the proof in three steps. The first two steps involve bootstrapping the solution to higher regularity -- increasing the number of spatial derivatives by approximately $1/2$ in each step. In the final step, we show that the solution is sufficiently regular and recover a solution of the CHR model (\ref{pde:CHRnoElastic}) for short time. 

Let $c\in H^{5,1+1/4}(\Omega_T)$ be a solution of the truncated CHR model from Theorem \ref{thm:CHRreg514}. For Step 1 and Step 2, we will assume that $c_0\in H^4(\Omega)$ satisfies $\partial_\nu c_0 = 0$ and $\partial_\nu (\Delta c_0) = \mathcal{R}(c_0,\Delta c_0)$ on $\Gamma$.

 \vspace{0.5em}

\noindent \textbf{Step 1.} Define 
\begin{equation}\label{def:Bk}
\mathcal{B}_k:=H^{4+k,1+k/4}(\Omega_T).
\end{equation} We show that $c \in \mathcal{B}_k$ for any $k \in (1,1+1/2)$.
Looking to Theorem \ref{thm:SPEwithBC} with $k\in (1,1+1/2)$, and noting control of the bulk term $(\Delta f')_\alpha (c)$ is direct, this step will be concluded if $\mathcal{R}(c,\Delta c) \in H^{1/2+k,(1/2+k)/4}(\Sigma_T)$. Let $Q:=(0,1)^3$ and $Q_0 :=(0,1)^{2}\times \{0\}$. Up to a local flattening and covering argument (which preserves all orders of regularity), this means we must show that $c \in H^{5,1+1/4}(Q \times (0,T))$ implies that 
\begin{equation}\label{eqn:spaceInclusionFin}
\mathcal{R}(c,\Delta c) \in H^{1/2+k,(1/2+k)/4}(Q_0 \times (0,T)).
\end{equation} We show 
\begin{equation}\label{eqn:spaceInclusion}
\mathcal{R}(c,\Delta c) \in H^{2,1/2}(Q_0 \times (0,T)),
\end{equation}
 which implies (\ref{eqn:spaceInclusionFin}) by Proposition \ref{prop:anisoInterp}. As before the challenge to prove (\ref{eqn:spaceInclusion}) is showing that $\nabla_{N-1}^2 \mathcal{R}(c,\Delta c) \in H^{0,0}(Q_0 \times (0,T))$, where $\nabla_{N-1}$ denotes the gradient with respect to the first $N-1 = 2$ variables. We use the Gagliardo-Nirenberg inequality (Theorem \ref{thm:nirenbergIneq}) in dimension $N-1=2$, with $i =1$ and $2$, to find 
\begin{equation}\label{bdd:gagBoundary}
\|\partial_i \Delta c\|_{L^4(Q_0)}^2 \leq \|\Delta c\|_{L^4(Q_0)}\| \Delta c\|_{W^{2,4}(Q_0)}.
\end{equation}
By Proposition \ref{prop:intermediateReg} and Theorem \ref{thm:LMnormalReg}, 
\begin{equation}\label{eqn:inclusion1}
\Delta c \in H^{2+1/2,5/8}(Q_0\times(0,T)),
\end{equation} and by the trace Theorem \ref{thm:traceSobo}, 
\begin{equation}\label{eqn:inclusion2}
\Delta c \in L^\infty (0,T; H^{1/2}(Q_0)).
\end{equation} Noting $H^{m+1/2}(Q_0) \hookrightarrow W^{m,4}(Q_0)$ for $m\in \N_0$ \cite{LeoniBook}, by (\ref{bdd:gagBoundary}), (\ref{eqn:inclusion1}), and (\ref{eqn:inclusion2}) we have 
\begin{equation}\label{bdd:badRterm}
\|(\partial_i \Delta c)^2\|_{H^{0,0}(Q_0\times (0,T))} \leq C\|\Delta c\|_{L^\infty (0,T; H^{1/2}(Q_0))} \|\Delta c\|_{H^{2+1/2,0}(Q_0\times (0,T))} <\infty.
\end{equation}
As in (\ref{eqn:RsecondOrder}), control of $\|\mathcal{R}(c,\Delta c)\|_{H^{2,1/2}(Q_0 \times (0,T))}$ follows from (\ref{bdd:badRterm}) and other bounds on lower order terms.

 \vspace{0.5em}
 
 \noindent \textbf{Step 2.} We prove that $c$ is in $H^{6,1+1/2}(\Omega_T)$.
 Looking to Theorem \ref{thm:SPEwithBC} and leaving the control of the (lower order) bulk term $(\Delta f')_\alpha (c)$ to further details, this step will be concluded if 
\begin{equation}\label{eqn:spaceInclusion2}
\mathcal{R}(c,\Delta c) \in H^{2+1/2,5/8}(\Sigma_T)
\end{equation} and the compatibility condition 
\begin{equation}\label{eqn:compat2}
\partial_\nu (\Delta c_0) (x) = \mathcal{R}(c,\Delta c)(x,0) \text{ for  } x \text{ a.e. in }  \Gamma
\end{equation} holds. We will show that $c\in \mathcal{B}_k$, as defined in (\ref{def:Bk}) for some $1<k<1+1/2$, implies $\mathcal{R}(c,\Delta c)\in H^{2+1/2,5/8}(\Sigma_T)$. By Proposition \ref{prop:intermediateReg} and as $\mathcal{R}$ is Lipschitz (see (\ref{ass:Rderivatives})), the necessary temporal regularity holds for any such $k$. Using continuity of the trace operator $ H^3(\Omega)\hookrightarrow H^{2+1/2}(\Gamma)$, spatial regularity will be complete if we prove the inclusion $\nabla ^3 \mathcal{R}(c,\Delta c) \in H^{0,0}(\Omega_T)$. Computing the derivative from (\ref{eqn:RsecondOrder}), and using Young's inequality, the nonlinear terms which need estimated in $H^{0,0}(\Omega_T)$ are $(\partial_i \Delta c)^3$ and $(\partial_i^2 \Delta c)^{3/2}$. First note that as $c \in \mathcal{B}_k,$ $\nabla^m c \in H^{0,k/4}(0,T;L^2(\Omega))$ for $m = 0,\ldots, 4.$ Using norm (\ref{def:besovAnisoNorm}), this implies that the norm $\|c\|_{H^4(\Omega)}$ has regularity given by $\|c\|_{H^4(\Omega)}\in H^{k/4}(0,T)\hookrightarrow L^z(0,T)$, with $z := \frac{4}{2-k}$ \cite{LeoniBook}.
Then by the Sobolev-Gagliardo-Nirenberg embedding theorem, we have
\begin{equation}\label{bdd:h6derivat}
\begin{aligned}
\int_0^T \|(\partial_i \Delta c)^{3}\|_{L^2(\Omega)}^{2}\, dt =& \int_0^T \|\partial_i \Delta c\|_{L^6(\Omega)}^{6}\, dt  \\
\leq & \int_0^T \| \Delta c\|_{H^4(\Omega)}^{6}\, dt < \infty,
\end{aligned}
\end{equation}
where in the last line, we have used that $z>6$ for $k$ sufficiently close to $1+1/2.$

By the Gagliardo-Nirenberg inequality (Theorem \ref{thm:nirenbergIneq}) with $r = 2$ and $q = 6$, Young's inequality, and (\ref{bdd:h6derivat}), we have
\begin{equation}\nonumber
\begin{aligned}
\int_0^T\|(\partial_i^2 \Delta c)^{3/2}\|_{L^2(\Omega)}^2\, dt =& \int_0^T \|\partial_i^2 \Delta c\|_{L^3(\Omega)}^3 \, dt  \\
\leq & C(\Omega) \int_\Omega \|c\|_{H^5(\Omega)}^{3/2}\| c\|_{W^{3,6}(\Omega)}^{3/2}\, dt \\
\leq & C(\Omega)\int_0^T\left(\|c\|_{H^{5}(\Omega)}^{2} + \| c\|_{W^{3,q}(\Omega)}^{6}\right) \, dt <\infty.
\end{aligned}
\end{equation}

Finally, to obtain the compatibility condition (\ref{eqn:compat2}), note $c\in C([0,T];H^3(\Omega))$, by Theorem \ref{thm:traceSobo}, implying $c$ and $\Delta c \in C([0,T];H^{1/2}(\Gamma))$ with $c(\cdot , 0) = c_0 $ and $\Delta c(\cdot,0) = \Delta c_0.$ By hypothesis on the initial condition, we have 
\begin{equation} \nonumber
 \partial_\nu (\Delta c_0)(x) = \mathcal{R}(c_0,\Delta c_0)(x) =  \mathcal{R}(c,\Delta c)(x,0)\quad \text{ for }x \text{ a.e. in }  \Gamma.
\end{equation}

\noindent \textbf{Step 3: Choosing truncations.}
We choose parameters in our definition of the truncated CHR model (\ref{pde:truncCHR}) so that the regular solution from Step 2 is a solution of the CHR model (\ref{pde:CHRnoElastic}). As $c_0 \in H^{4}(\Omega)$, $\|c_0\|_{C^2(\Omega)}=: \alpha<\infty$ by the Sobolev-Gagliardo-Nirenberg and Morrey embedding theorems. Let $0< \alpha := \|c_0\|_{C^2(\Omega)} +1 <\infty$. Then by (\ref{def:Rfp}), we have $\partial_\nu (\Delta c_0)  = -R(c_0,-\Delta c_0 +f'(c_0)) = \mathcal{R}(c_0,\Delta c_0)$ on $\Gamma$. Applying Steps 1 and 2, there is a solution $c\in H^{6,1+1/2}(\Omega_{T_0})$ of the truncated CHR model (\ref{pde:truncCHR}) for some $T_0>0$.  By the trace Theorem \ref{thm:traceSobo} and the Sobolev-Gagliardo-Nirenberg and Morrey embedding theorems in dimension $N=3$, $c \in C([0,T]; C^{2,a}(\Omega))$ for some $a>0$. By continuity, there is $T>0$ for which $\mathcal{R}(c,\Delta c  ) = -R(c,-\Delta c + f'(c))$ on $\Sigma_T$ and $(\Delta \tilde{f})_{\alpha} (c) = \Delta f(c)$ on $\Omega_T$, proving the theorem.
\end{proof}

\begin{remark}
We note there are initial conditions $c_0$ which satisfy the hypothesis of Theorem \ref{thm:CHRexist3D}. 
Recall, as in Singh et al. \cite{singh2008intercalation}, by (\ref{def:Rbazant}), we have that
\begin{equation}\nonumber 
R(c,\mu) = R_{\rm{ins}}-R_{\rm{ext}}= k_{\rm{ins}}\exp(\beta (\mu_e-\mu)) - k_{\rm{ext}}c\exp(\beta(\mu-\mu_e)),
\end{equation}
where all constants $k_{\rm{ext}}, \ k_{\rm{ins}}, \ \beta, \ \mu_e$ are positive. Consider the case of a constant $c_0\in (0,1)$, then we have 
\begin{equation}\nonumber
\nonumber
R(c_0,-\Delta c_0 +f'(c_0)) = R(c_0,f'(c_0)) = k_{\rm{ins}}\exp(\beta (\mu_e-f'(c_0))) - k_{\rm{ext}}c_0\exp(\beta(f'(c_0)-\mu_e)).
\end{equation}
Since $\lim\limits_{z\to 0}f'(z) = -\infty$ and $\lim\limits_{z\to 1}f'(z) = \infty$, it follows that 
\begin{equation}\nonumber
\lim_{c_0 \to 0}R(c_0,f'(c_0)) = \infty, \quad \quad \lim_{c_0 \to 1}R(c_0,f'(c_0)) = -\infty.
\end{equation} By the intermediate value theorem, there is $c_0\in (0,1)$ such that $R(c_0,-\Delta c_0 + f'(c_0)) = 0.$ It then follows that $c_0$ is an admissible condition for Theorem \ref{thm:CHRexist3D} as $\partial_\nu (\Delta c_0) = 0 = -R(c_0,-\Delta c_0 +f'(c_0)).$ Considering interior perturbations of $c_0,$ we may find other admissible initial conditions. 

\end{remark}

\section{A Priori Estimates: Proof of Theorem \ref{thm:CHRcont}}\label{sec:aprioriEst}

Two challenges occur which make the proof of Theorem \ref{thm:CHRcont} involved:
\begin{itemize}
\item We know that $\|v\|_{L^\infty(0,T;H^2(\Omega))} \leq C \|v\|_{H^{4,1}(\Omega_T)}$ from Theorem \ref{thm:traceSobo}. Necessarily though, $C$ depends on $T$, blowing up as $T\to 0$ (consider a function constant in time). 

\item Returning to the notation of Subsection \ref{subsec:anisoSobolev}, there is insufficient literature detailing the constants by which $\|\cdot\|_{H^s(0,T),I}$ is equivalent to $\|\cdot \|_{H^{s}(0,T)}$, where the later norm is given by the integral of the derivative or difference quotients. Existing results of which the authors are aware address this relation with the use of extensions (see, e.g., \cite{chandler-interpSobolev}).
\end{itemize}
As we will send $T\to 0,$ i.e., shrink the size of our domain, these constants are critical. To navigate this problem, the first trick in our argument is to extend the functions/data on the right hand side of (\ref{pde:truncCHR}) to the domain $\Omega_1$ using Corollary \ref{cor:fracExtension}. Then we may use regularity theory for parabolic equations (see Theorem \ref{thm:SPEwithBC}) on a domain independent of $T$. Any dependence on $T>0$ in the estimates arises from the extension in Corollary \ref{cor:fracExtension}, which allows for precise control. 

\begin{proof}[Proof of Theorem \ref{thm:CHRcont}]

We prove the result in two steps: first $k=0$, and then $k=1/2$. 

\noindent \textbf{Step 1: $\boldsymbol{k=0}$.}
Let $c$ be a strong solution of the truncated CHR model (\ref{pde:truncCHR}) on $\Omega_T$ for $0<T\leq 1.$
We translate the initial condition to consider a PDE for $b := c- \tilde c_0$
\begin{equation}\label{pde:CHRtrans}
\left\{\begin{aligned}
&\partial_t b +\Delta^2 b  =(\Delta f')_\alpha(c)-\Delta^2\tilde c_0 =:g  & \text{in } \Omega_T,\\
&\partial_\nu b  = 0 & \text{on } \Sigma_T,\\
&\partial_\nu (\Delta b)  = [\mathcal{R}(c,\Delta c )-\mathcal{R}(\tilde c_0,\Delta \tilde c_0)] +[\mathcal{R}(\tilde c_0,\Delta \tilde c_0) - \partial_\nu (\Delta \tilde c_0)] =:h & \text{on } \Sigma_T,\\
& b(0)= c_0 - \tilde c_0 & \text{in } \Omega.
\end{aligned}\right.
\end{equation} 
Note that as $c\in H^{4,1}(\Omega_T)$, $h \in H^{1/2,1/8}(\Sigma_T)$ (see Theorem \ref{thm:LMnormalReg}) and by Corollary \ref{cor:fracExtension} with (\ref{def:besovAnisoNorm}) may be extended to $\tilde{h}\in H^{1/2,1/8}(\Sigma_1)$ satisfying the bound 
\begin{equation}\nonumber
\|\tilde{h}\|_{H^{1/2,1/8}(\Sigma_1)}\leq C\Big((1+T^{-1/8})\|h \|_{H^{0,0}(\Sigma_T)} + \|h\|_{H^{1/2,1/8}(\Sigma_T)}\Big).
\end{equation}
Let $\tilde g$ be the extension of $g \in H^{0,0}(\Omega_T)$ to $H^{0,0}(\Omega_1)$ defined by Corollary \ref{cor:fracExtension}.

We consider the PDE for $\bar b$ on the extended domain $\Omega_1:$
\begin{equation}\label{pde:CHRtransExt}
\left\{\begin{aligned}
& \partial_t \bar b + \Delta^2 \bar b =\tilde g & \text{in } \Omega_1, \\
& \partial_\nu \bar b  = 0  & \text{on }\Sigma_1,\\
& \partial_\nu (\Delta \bar b) = \tilde{h}  & \text{on }\Sigma_1 ,\\
& \bar b(0) = c_0 - \tilde c_0 & \text{in } \Omega .
\end{aligned}\right.
\end{equation} 
If problem (\ref{pde:CHRtransExt}) admits a solution $\bar b \in H^{4,1}(\Omega)$, then as (\ref{pde:CHRtransExt}) coincides with (\ref{pde:CHRtrans}) on $\Omega_T$, by uniqueness (see Theorem \ref{thm:SPEwithBC}), $\bar b|_{\Omega_T} = b.$

We apply Theorem \ref{thm:SPEwithBC} (with $k=0$) to conclude that (\ref{pde:CHRtransExt}) admits a unique solution $\bar b$ satisfying:
\begin{equation}\label{bdd:barCterms}
\begin{aligned}
\|\bar b\|_{H^{4,1}(\Omega_1)} & + \|\bar b\|_{L^\infty(0,1;H^2(\Omega))}\\
\leq & C(\Omega,1)\Big(\| c_0 -\tilde c_0\|_{H^{2}(\Omega)}  + \|\tilde g\|_{H^{0,0}(\Omega_1)}+\|\tilde{h}\|_{H^{1/2,1/8}(\Sigma_1)} \Big) \\
\leq & C\Big(\| c_0 -\tilde c_0\|_{H^{2}(\Omega)} + \|g\|_{H^{0,0}(\Omega_T)} +  (1+T^{-1/8})\|h\|_{H^{0,0}(\Sigma_T)} + \|h\|_{H^{1/2,1/8}(\Sigma_T)}\Big) \\
 =: & C(\| c_0 -\tilde c_0\|_{H^{2}(\Omega)} + A_1 + A_2 +A_3).
\end{aligned}
\end{equation}
We note that $C$ in the above estimate is independent of $T$, and the extension $\Omega_T$ to $\Omega_1$ was specifically done to control dependence of constants on $T$ in the above expression.
We now estimate each term in the above expression.

\noindent \textbf{Term $\boldsymbol{A_1}$:}
As the truncated Laplacian (\ref{eqn:DeltafTrunc}) is bounded in $L^\infty$ and $\tilde c_0$ is constant in time, we have
\begin{equation}\label{bdd:termI}
A_1 = \|g\|_{H^{0,0}(\Omega_T)} \leq  T^{1/2}C(\tilde c_0, \alpha).
\end{equation}

\noindent \textbf{Term $\boldsymbol{A_2}$:}
We note that $H^{3}(\Omega) = [H^{4}(\Omega),H^{2}(\Omega)]_{1/2},$ so $\|\cdot \|_{H^3(\Omega)} \leq C(\Omega) \|\cdot \|_{H^2(\Omega)}^{1/2} \|\cdot \|_{H^4(\Omega)}^{1/2}$ for some constant $C(\Omega)>0$ (see \cite{LM-v1}). It follows by H\"older's inequality that 
\begin{equation}\label{bdd:addendum}
\begin{aligned}
\|c-\tilde c_0\|_{H^{3,0}(\Omega_T)}^2 \leq & C(\Omega) \int_0^T  \|c(t)-\tilde c_0\|_{H^2(\Omega)} \|c(t) -\tilde c_0\|_{H^4(\Omega)} \ dt \\
\leq & C(\Omega) \|c-\tilde c_0\|_{H^{2,0}(\Omega_T)}\|c-\tilde c_0\|_{H^{4,0}(\Omega_T)} \\
\leq & C(\Omega) T^{1/2} \|c-\tilde c_0\|_{L^\infty(0,T;H^2(\Omega))}\|c-\tilde c_0\|_{H^{4,0}(\Omega_T)} \\
\leq & C(\Omega) T^{1/2} \Big( \|c-\tilde c_0\|_{L^\infty(0,T;H^2(\Omega))}^2 + \|c-\tilde c_0\|_{H^{4,0}(\Omega_T)}^2\Big).
\end{aligned}
\end{equation}
Using continuity of the trace operator (see \cite{LeoniBook}), $\|h(\cdot,t)\|_{H^{1/2}(\Gamma)}\leq C(\Omega)\|h(\cdot,t)\|_{H^1(\Omega)}$ for $t$-a.e. in $(0,T)$. Hence, 
\begin{equation}\label{bdd:hTrToSpace}
\|h\|_{H^{1/2,0}(\Sigma_T)}\leq \|h\|_{H^{1,0}(\Omega_T)}.
\end{equation}
By the definition of $h$ in (\ref{pde:CHRtrans}), we have
\begin{equation}\label{bdd:termIIa}
\begin{aligned}
\|h\|_{H^{1,0}(\Omega_T)}  \leq & C(\Omega) \left(\|\mathcal{R}(c,\Delta c )-\mathcal{R}(\tilde c_0,\Delta \tilde c_0)\|_{H^{1,0}(\Omega_T)} + \sqrt{T} C(\tilde c_0)  \right).
\end{aligned}
\end{equation}
To estimate the first term of the right hand side, we use that $\mathcal{R}$ is bi-Lipschitz (from (\ref{ass:Rderivatives})) to control the $H^{0,0}$ norm by $\|c-\tilde c_0\|_{H^{2,0}(\Omega_T)}$. Using the chain rule and rearranging, the gradient is given by
\begin{equation}\label{eqn:difference}
\begin{aligned}
& \nabla (\mathcal{R}(c,\Delta c )-\mathcal{R}(\tilde c_0,\Delta \tilde c_0)) = \\
& \partial_s \mathcal{R}(c,\Delta c)\nabla c - \partial_s \mathcal{R}(\tilde c_0,\Delta \tilde c_0)\nabla \tilde c_0 +\partial_w \mathcal{R}(c,\Delta c)\nabla (\Delta c)-\partial_w \mathcal{R}(\tilde c_0,\Delta \tilde c_0)\nabla (\Delta \tilde c_0).
\end{aligned}
\end{equation}
From (\ref{eqn:difference}) and (\ref{ass:Rderivatives}), we estimate
\begin{equation}\label{bdd:Rdifference}
\begin{aligned}
\|\partial_w \mathcal{R}(c,\Delta c)\nabla (\Delta c) & - \partial_w \mathcal{R}(\tilde c_0,\Delta \tilde c_0)\nabla (\Delta \tilde c_0)\|_{L^{2}(\Omega)} \\
 \leq & \|\partial_w \mathcal{R}(c,\Delta c)\left( \nabla (\Delta c) - \nabla (\Delta \tilde c_0)\right)\|_{L^{2}(\Omega)} \\
& + \|\left(\partial_w \mathcal{R}(c,\Delta c) - \partial_w \mathcal{R}(\tilde c_0,\Delta \tilde c_0)\right) \nabla (\Delta \tilde c_0) \|_{L^2(\Omega)} \\
\leq & C(\mathcal{R})\left( \|c-\tilde c_0\|_{H^3(\Omega)} + \|\nabla (\Delta \tilde c_0)\|_{L^2(\Omega)}\right) .
\end{aligned}
\end{equation}
By a similar bound for the first difference in (\ref{eqn:difference}), we use (\ref{bdd:addendum}), (\ref{bdd:termIIa}), and (\ref{bdd:Rdifference}) to conclude
\begin{equation}\label{bdd:termIIb}
\begin{aligned}
\|h\|_{H^{1,0}(\Omega_T)} \leq & \, C(\mathcal{R},\Omega) T^{1/4} \Big(\|c-\tilde c_0\|_{L^\infty(0,T;H^2(\Omega))} + \|c - \tilde c_0 \|_{H^{4,0}(\Omega_T)} + T^{1/4} C(\tilde c_0)\Big) .
\end{aligned}
\end{equation}
By the definition of $A_2$ in (\ref{bdd:barCterms}), (\ref{bdd:hTrToSpace}), and (\ref{bdd:termIIb})
\begin{equation}\label{bdd:termII}
\begin{aligned}
A_2  \leq & C(\mathcal{R},\Omega) T^{1/8} \Big(\|c-\tilde c_0\|_{L^\infty(0,T;H^2(\Omega))} + \|c - \tilde c_0 \|_{H^{4,0}(\Omega_T)} + T^{1/4} C(\tilde c_0)\Big) .
\end{aligned}
\end{equation}

\noindent \textbf{Term $\boldsymbol{A_3}$:}
First, we focus our attention on control of the semi-norm $|h|_{H^{0,1/8}(\Sigma_T)}.$
For convenience, we will define $\mathcal{R}_c:= \mathcal{R}(c,\Delta c).$ Setting $(\mathcal{R}_c)_T(x,t):=\mathcal{R}_c(x,Tt)$ and $c_T(x,t):=c(x,Tt)$, using Lemma \ref{lem:fracNormStretch}, Theorem \ref{thm:LMnormalReg}, Proposition \ref{prop:intermediateReg}, (\ref{ass:Rderivatives}), and that $\tilde c_0$ is independent of time, it follows that 
\begin{equation}\nonumber
\begin{aligned}
|h|_{H^{0,1/8}(\Sigma_T)}  = & \, |\mathcal{R}_c |_{H^{0,1/8}(\Sigma_T)}\\
= & \, T^{3/8}|(\mathcal{R}_c)_T|_{H^{0,1/8}(\Sigma_1)} \\
= & \, T^{3/8}|\mathcal{R}_{(c_T)}|_{H^{0,1/8}(\Sigma_1)} \\
\leq & \, C(\mathcal{R}) T^{3/8}\left(|(c-\tilde c_0)_T|_{H^{0,1/8}(\Sigma_1)}+ |\Delta (c-\tilde c_0)_T|_{H^{0,1/8}(\Sigma_1)}\right) \\
\leq &\, C T^{3/8}\|(c-\tilde c_0)_T\|_{H^{3,3/4}(\Omega_1)} \\
\leq  &\, C T^{3/8}\left(\|(c-\tilde c_0)_T\|_{H^{3,0}(\Omega_1)} +|(c-\tilde c_0)_T|_{H^{0,3/4}(\Omega_1)}\right).
\end{aligned}
\end{equation}
Using a change of variables, we have that $\|(c-\tilde c_0)_T\|_{H^{3,0}(\Omega_1)} =T^{-1/2}\|(c-\tilde c_0)\|_{H^{3,0}(\Omega_T)}.$ By Proposition \ref{prop:BesovEst} and a change of variables, we have
\begin{equation}\nonumber
\begin{aligned}
|(c-\tilde c_0)_T|_{H^{0,3/4}(\Omega_1)}\leq & \, C\|\partial_t ((c-\tilde c_0)_T)\|_{H^{0,0}(\Omega_1)}   \\
= &  \, C T \|(\partial_t (c-\tilde c_0))_T\|_{H^{0,0}(\Omega_1)} = CT^{1/2}\|\partial_t (c-\tilde c_0)\|_{H^{0,0}(\Omega_T)}.
\end{aligned}
\end{equation}
Consolidating these estimates along with (\ref{bdd:addendum}), we find
\begin{equation}\label{bdd:termIIIa}
|h|_{H^{0,1/8}(\Sigma_T)} \leq  C(\mathcal{R} ,\Omega) T^{1/8} \left( \|c-\tilde c_0\|_{H^{4,1}(\Omega_T)} +\|c-\tilde c_0\|_{L^{\infty}(0,T;H^2(\Omega))} \right).
\end{equation}
By (\ref{bdd:hTrToSpace}), (\ref{bdd:termIIb}), and (\ref{bdd:termIIIa}), we have 
\begin{equation}\label{bdd:termIII}
A_3 \leq  C(\mathcal{R} ,\Omega) T^{1/8} \left( \|c-\tilde c_0\|_{H^{4,1}(\Omega_T)} +\|c-\tilde c_0\|_{L^{\infty}(0,T;H^2(\Omega))}  + T^{1/4} C(\tilde c_0)\right).
\end{equation}

Returning to (\ref{bdd:barCterms}), using Theorem \ref{thm:traceSobo}, recalling $b := c - \tilde c_0,$ and combining the bounds (\ref{bdd:termI}), (\ref{bdd:termII}), and (\ref{bdd:termIII}), we find
\begin{equation}\nonumber
\begin{aligned}
 \| c & -\tilde c_0  \|_{H^{4,1}(\Omega_T)} + \|c-\tilde c_0\|_{L^\infty(0,T;H^2(\Omega))} \\
 \leq & \|\bar b\|_{H^{4,1}(\Omega_1)} + \| \bar b\|_{L^\infty(0,1;H^2(\Omega))} \\
\leq & C\| c_0 -\tilde c_0\|_{H^2(\Omega)}  + C( c_0 ,\tilde c_0 )T^{1/8} \left(\| c-\tilde c_0 \|_{H^{4,1}(\Omega_T)} + \|(c-\tilde c_0)\|_{L^\infty(0,T;H^2(\Omega))}\right) +T^{3/8}C(c_0 ,\tilde c_0) .
\end{aligned}
\end{equation}
With this inequality, the theorem holds for $k=0$ and $T_0 = \frac{1}{2C(c_0,\tilde c_0)^8}$. For use in the next step, note that the above inequality implies
\begin{equation}\label{bdd:CHRcontExt}
\|\bar b\|_{H^{4,1}(\Omega_1)} \leq  C\| c_0-\tilde c_0\|_{H^{2}(\Omega)} + \eta(c_0, \tilde c_0,T).
\end{equation}

\noindent \textbf{Step 2: $\boldsymbol{k=1/2}$.} Note a solution for $k=1/2$ is also a solution for $k=0$, i.e., belonging to $H^{4,1}(\Omega_T)$. Consequently, we consider a solution $\bar b$ of (\ref{pde:CHRtransExt}) that satisfies (\ref{bdd:CHRcontExt}). Furthermore by Corollary \ref{cor:fracExtension} and (\ref{def:besovAnisoNorm}), we have
\begin{equation}\label{bdd:gExt}
\|\tilde g\|_{H^{1/2,1/8}(\Omega_1)} \leq C(1+T^{-1/8}) \|g\|_{H^{0,0}(\Omega_T)} + C\|g\|_{H^{1/2,1/8}(\Omega_T)}
\end{equation} 
and 
\begin{equation}\label{bdd:hExt}
\|\tilde h\|_{H^{1,1/4}(\Sigma_1)} \leq C(1+T^{-1/4}) \|h\|_{H^{0,0}(\Sigma_T)} + C\|h\|_{H^{1,1/4}(\Sigma_T)}.
\end{equation} 

Recalling bound (\ref{bdd:termI}), we have $\|g\|_{H^{0,0}(\Omega_T)} \leq C(\tilde c_0, \alpha)\sqrt{T}.$ To control the fractional norm in space arising in (\ref{bdd:gExt}), we simply control $g \in H^{1,0}(\Omega_T)$. By (\ref{eqn:DeltafTrunc}) and (\ref{pde:CHRtrans}), up to terms constant in time (controlled in energy by $\sqrt{T}$), we must control the derivative of $(\Delta f')_\alpha(c) - (\Delta f')_\alpha(\tilde c_0)$. To see how this is done, we bound one term of the derivative, e.g., $ f_\alpha''(c)\psi_\alpha'(\Delta c)\nabla (\Delta c) - f_\alpha''(\tilde c_0)\psi_\alpha'(\Delta \tilde c_0)\nabla (\Delta \tilde c_0):$
\begin{equation}\label{bdd:falpha}
\begin{aligned}
\|f_\alpha''(c)& \psi_\alpha'(\Delta c)\nabla (\Delta c) - f_\alpha''(\tilde c_0)\psi_\alpha'(\Delta \tilde c_0)\nabla (\Delta \tilde c_0)\|_{L^2(\Omega)} \\
\leq & \|f_\alpha''(c) \psi_\alpha'(\Delta c)\left[ \nabla (\Delta c) - \nabla (\Delta \tilde c_0)\right]\|_{L^2(\Omega)} + \|\left[f_\alpha''(c) \psi_\alpha'(\Delta c)- f_\alpha''(\tilde c_0)\psi_\alpha'(\Delta \tilde c_0)\right] \nabla (\Delta \tilde c_0)\|_{L^2(\Omega)} \\
\leq & C\|\nabla (\Delta c) - \nabla (\Delta \tilde c_0)\|_{L^2(\Omega)} + C(\tilde c_0, f_\alpha, \psi_\alpha, \Omega),
\end{aligned}
\end{equation}
where we used the fact that $f_\alpha \in C^5(\R)$ and $\psi_\alpha \in C^\infty_c(\R).$
Using the bound (\ref{bdd:CHRcontExt}), we can control the right hand side of (\ref{bdd:falpha}) by $C\| c_0-\tilde c_0\|_{H^{2}(\Omega)} + \eta(c_0, \tilde c_0,T)$ when we integrate in time. Splitting the difference of a product via the standard trick and noting $\tilde c_0$ is independent of time, using Theorem \ref{prop:intermediateReg}, it is straightforward to show $|g|_{H^{0,1/8}(\Omega_T)} \leq C\|c- \tilde c_0\|_{H^{5/2,5/8}(\Omega_T)}+C\sqrt{T}$. In total, using (\ref{bdd:CHRcontExt}) again, we have
\begin{equation}\nonumber
\|\tilde g\|_{H^{1/2,1/8}(\Omega_1)} \leq C\| c_0-\tilde c_0\|_{H^{2}(\Omega)} + \eta(c_0, \tilde c_0,T).
\end{equation} 

Control of the first term on the right-hand side of (\ref{bdd:hExt}) follows from (\ref{bdd:termII}) and the definition of $A_2$ in (\ref{bdd:barCterms}), while the high order term $\|h\|_{H^{1,1/4}(\Sigma_T)}$ is controlled using (\ref{bdd:CHRcontExt}) by 
\begin{equation}\nonumber
\begin{aligned}
|\mathcal{R}(c, \Delta c)-\mathcal{R}(\tilde c_0,\Delta \tilde c_0)|_{H^{1,1/4}(\Sigma_T)}  \leq & \,  |\mathcal{R}(c,\Delta c)|_{H^{1,1/4}(\Sigma_T)} + |\mathcal{R}(\tilde c_0,\Delta \tilde c_0)|_{H^{1,1/4}(\Sigma_T)} \\
\leq &\, C(\mathcal{R}) \left(|c|_{H^{1,1/4}(\Sigma_T)} + |\Delta c|_{H^{1,1/4}(\Sigma_T)}\right) + \sqrt{T}C(\tilde c_0) \\
\leq & \,  C \left(|c - \tilde c_0|_{H^{1,1/4}(\Sigma_T)} +  |\Delta c- \Delta \tilde c_0|_{H^{1,1/4}(\Sigma_T)}\right) + \sqrt{T}C(\tilde c_0) \\
\leq & \, C \left(|\bar b|_{H^{1,1/4}(\Sigma_1)} +  |\Delta \bar b|_{H^{1,1/4}(\Sigma_1)}\right) + \sqrt{T}C(\tilde c_0) \\
\leq & \, C \|\bar b\|_{H^{3+1/2,7/8}(\Omega_1)} +\sqrt{T}C(\tilde c_0) \\
\leq & \, C\| c_0-\tilde c_0\|_{H^{2}(\Omega)} + \eta(c_0, \tilde c_0,T),
\end{aligned}
\end{equation}
where we have used (\ref{ass:Rderivatives}), Theorem \ref{thm:LMnormalReg}, and that if a function $F$ is Lipschitz continuous, then $|F(v)|_{H^1(\Gamma)} \leq C(F) (|v|_{H^1(\Gamma)}+1)$, which follows from a flattening argument. As in Step 1, these estimates, in conjunction with Theorem \ref{thm:SPEwithBC} (for $k=1/2$) applied to (\ref{pde:CHRtransExt}) and the trace Theorem \ref{thm:traceSobo}, conclude the theorem.
\end{proof}

\section*{Acknowledgments} This paper is part of the author's Ph.D. thesis at Carnegie Mellon University under the direction of Irene Fonseca and Giovanni Leoni. The author is deeply indebted to these two for expert guidance on many mathematical topics and for their many hours spent watching the author point at PDFs with an emphatic cursor. Furthermore, the author is thankful for their many helpful remarks on the organization of the paper and spotting a variety of typos, which lead to an immensely improved paper. The author was partially supported by National Science Foundation Grants DMS-1411646, DMS-1714098, and DMS-1906238.

\bibliographystyle{amsplain}
\bibliography{kstinson_CHRfp_3d}

\end{document}